\documentclass[times]{nlaauthmod}

\usepackage[bigdelims,cmintegrals,vvarbb]{newtxmath}
\usepackage[hidelinks,bookmarksopen,bookmarksnumbered]{hyperref}

\usepackage{subfig,multirow}
\usepackage{mathtools}
\usepackage{bm}
\usepackage{enumitem}

\input{defines.texx}
\allowdisplaybreaks[4]

\numberwithin{equation}{section}

\usepackage[noabbrev,capitalize]{cleveref}

\newtheorem{theorem}{Theorem}[section]
\newtheorem{proposition}[theorem]{Proposition}
\newtheorem{lemma}[theorem]{Lemma}

\newtheorem{definition}[theorem]{Definition}
\newtheorem{remark}[theorem]{Remark}

\crefname{equation}{}{}
\crefname{subsection}{subsection}{subsections}
\crefname{section}{section}{sections}

\newcommand{\email}[1]{\protect\href{mailto:#1}{#1}}

\renewcommand{\star}{*}

\begin{document}

\runningheads{D. Z. Kalchev, T. A. Manteuffel and S. M{\"u}nzenmaier}{\texorpdfstring{$(\cL\cL^\star)^{-1}$}{(LL*)\^{}\{-1\}} and \texorpdfstring{$\cL\cL^\star$}{LL*} least-squares finite element methods}

\title{Mixed $(\cL\cL^\star)^{-1}$ and $\cL\cL^\star$ least-squares finite element methods with application to linear hyperbolic problems}

\author{Delyan Z. Kalchev\affil{1}\corrauth, Thomas A. Manteuffel\affil{1} and Steffen M{\"u}nzenmaier\affil{1}\affil{2}}

\address{\affilnum{1}Department of Applied Mathematics, University of Colorado at Boulder, USA
         \break
         \affilnum{2}Fakult{\"a}t f{\"u}r Mathematik, Universit{\"a}t Duisburg-Essen, Germany}

\corraddr{Delyan Kalchev, Department of Applied Mathematics, 526 UCB, University of Colorado at Boulder, Boulder, CO 80309-0526, USA. Email: \email{delyan.kalchev@colorado.edu}}

\cgs{This work was performed under the auspices of the U.S. Department of Energy under grant numbers (SC) DE-FC02-03ER25574 and (NNSA) DE-NA0002376, Lawrence Livermore National Laboratory under contract B614452.}

\begin{abstract}
In this paper, a few dual least-squares finite element methods and their application to scalar linear hyperbolic problems are studied. The purpose is to obtain $L^2$-norm approximations on finite element spaces of the exact solutions to hyperbolic partial differential equations of interest. This is approached by approximating the generally infeasible quadratic minimization, that defines the $L^2$-orthogonal projection of the exact solution, by feasible least-squares principles using the ideas of the original $\cL\cL^\star$ method proposed in the context of elliptic equations. All methods in this paper are founded upon and extend the $\cL\cL^\star$ approach which is rather general and applicable beyond the setting of elliptic problems. Error bounds are shown that point to the factors affecting the convergence and provide conditions that guarantee optimal rates.
Furthermore, the preconditioning of the resulting linear systems is discussed. Numerical results are provided to illustrate the behavior of the methods on common finite element spaces.
\end{abstract}


\maketitle

\section{Introduction}

Consider a scalar linear hyperbolic partial differential equation (PDE) of the form
\begin{equation}\label{eq:pde}
\begin{alignedat}{3}
\gradd \bb \psi + \sigma \psi &= r &&\quad\text{in }\Omega,\\
\psi &= g &&\quad\text{on }\Gamma_I,
\end{alignedat}
\end{equation}
where the simply connected domain $\Omega \subset \bbR^d$ ($d$ is the dimension of the Euclidean space), flow field $\bb \in [L^\infty(\Omega)]^d$, absorption coefficient $\sigma \in L^\infty(\Omega)$, source $r\in L^2(\Omega)$, and inflow boundary data\footnote{In general, the function $g$ is in a space on $\Gamma_I$ that can be larger than $L^2(\Gamma_I)$; see the trace results in \cite{2004LinearHyperbolic}. For our considerations, the space $L^2(\Gamma_I)$ is sufficiently rich for inflow boundary conditions.} $g \in L^2(\Gamma_I)$ are given and $\psi$ is the unknown dependent variable. Here, $\Gamma_I$ denotes the inflow portion of the boundary $\partial \Omega$, $\Gamma_I = \set{\bx \in \partial \Omega \where \bn(\bx)\cdot\bb(\bx) < 0}$, where $\bn$ is the unit outward normal to $\partial \Omega$.

Equations like \eqref{eq:pde} arise often in applications and can also serve as model problems towards solving more elaborate hyperbolic PDEs \cite{LewisTransport,LeVequeHCL,LeVequeHyperbolicPDE,GodlewskiHCL,LaxHyperbolicPDE,2000BoundaryTransport}.

The solution to \eqref{eq:pde} can be quite irregular -- exhibiting jump discontinuities or, depending on the contrast in $\sigma$, extremely steep exponential layers leading to large variations of the solution in neighboring subregions of $\Omega$. We are interested in obtaining approximations of the solution without utilizing any additional information on its features and using only information provided by the differential operator in \eqref{eq:pde}. In particular, we consider general unstructured meshes that are not aligned with the flow, $\bb$, i.e., the mesh does not follow the characteristics of \eqref{eq:pde}. Also, the mesh does not need to resolve steep exponential layers, i.e., on the scale of the mesh such layers can appear as jump discontinuities. Moreover, we aim at solving \eqref{eq:pde} as a global space-time problem (if one of the independent variables represents time) without applying any time-stepping scheme, i.e., $\Omega$ is a domain in the space-time.

Least-squares finite element methods have been extensively studied for problems of elliptic and parabolic types; see, e.g., \cite{BochevLSFEM,1998LSFEM,1994FOSLS1,1997FOSLS2,JiangLSFEM,1995FOSLSStokes1,1997FOSLSStokes2,1998FOSLSNS1,1999FOSLSNS2}. They have also been applied to hyperbolic problems, including of the type \eqref{eq:pde}; cf., \cite{2004LinearHyperbolic,2001ErrorLS,2001ComparativeLS,1988LSHypSystems}, see also \cite{OlsonPhDthesis,2005HdivHyperbolic,2005LSShallowWater,2002DGLSStab}. These methods exhibit substantial numerical dissipation, unless proper scaling is implemented, which may include utilizing information about the characteristics of the problem and the respective features of the solution \cite{clsscaling}. Dissipation results in stable methods and least-squares have been used to augment Galerkin formulations to stabilize them; see, e.g, \cite{2002DGLSStab}. However, excessive dissipation can lead to unsatisfactory quality of the approximation. In our experience, this especially holds when large jumps in $\sigma$ cause very steep exponential layers in the solution that are not resolved by the mesh.

In this paper, we address these issues (the solution irregularity, unstructured meshes not resolving steep exponential layers, and the excessive numerical dissipation) by seeking approximations in the $L^2(\Omega)$ norm. Note that the least-squares methods \cite{2004LinearHyperbolic,2001ErrorLS} possess coercivity in a norm stronger than the $L^2(\Omega)$ norm, so they control the $L^2$-norm error but it can remain relatively large until the mesh size is sufficiently small to begin resolving the features of the solution. This contributes to the amount of numerical dissipation in the least-squares methods. In contrast, we approach the $L^2$-norm approximation more directly. The $(\cL\cL^\star)^{-1}$ and $\cL\cL^\star$ methods considered in this paper are based on least-squares principles, which, in a sense, approximate the minimization that defines the best $L^2$-norm approximation. 

Generally, given $f\in L^2(\Omega)$ and a linear first-order differential operator, $L$, our goal is to solve an equation of the form
\begin{equation}\label{eq:opeq}
L u = f,
\end{equation}
for the unknown $u\in \cD(L)$, where $\cD(L)$ denotes the domain of $L$. The general definition of $\cD(L)$ is provided in \cref{sec:notation} and in \cref{sec:application} the particular definition for \eqref{eq:pde} is shown. Equation \eqref{eq:pde} can be reduced to \eqref{eq:opeq} using superposition, since, in this case, the functions in $\cD(L)$ vanish on $\Gamma_I$. In practice, solving \eqref{eq:opeq} is addressed by numerically approximating the exact solution, $\hu \in \cD(L)$, of equation \eqref{eq:opeq}. The focus of this paper is on obtaining finite element approximations of $\hu$ with respect to the $L^2(\Omega)$ norm, denoted $\lV\cdot\rV$. Given a finite element space $\cU^h$, the best $L^2$-norm approximation of $\hu$ is defined by the minimization
\begin{equation}\label{eq:L2min}
u^h = \argmin_{v^h \in \cU^h} \lV v^h - \hu \rV^2,
\end{equation}
where the minimizer, $u^h$, is the $L^2$-orthogonal projection of $\hu$ onto $\cU^h$. The minimization problem \eqref{eq:L2min} can be reformulated as a standard $\cL\cL^\star$ method \cite{2001FOSLLstar,2005FOSLLstarGeneral}, but only for a special choice of the finite element space. However, for general $\cU^h$, the $L^2$-orthogonal projection of $\hu$ onto $\cU^h$ cannot be directly computed (unless the exact solution, $\hu$, is readily known). The idea here is to replace \eqref{eq:L2min} with a similar, but computationally feasible, minimization problems using an additional (auxiliary) finite element space and applying the ideas of the standard $\cL\cL^\star$ and negative-norm methods; see, e.g., \cite{1997NegFOSLS} for an $H^{-1}$ approach to elliptic problems. In comparison, the standard $\cL\cL^\star$ method obtains the best $L^2(\Omega)$ approximation under the compromise of using a particular and nonstandard finite element space, whereas the $(\cL\cL^\star)^{-1}$ and $\cL\cL^\star$-type methods studied in this paper allow utilizing standard finite element spaces but generally do not provide precisely the $L^2$-orthogonal projection of the exact solution.

Several methods are studied and compared in this paper. In particular, the $(\cL\cL^\star)^{-1}$ method is in the class of negative-norm least-squares methods. However, unlike a more standard $H^{-1}$ approach, the $(\cL\cL^\star)^{-1}$ method is better tailored to the particular problem \eqref{eq:opeq}. Namely, the isomorphism $(-\lap)^{-1}$ in the $H^{-1}$ method is replaced\footnote{In view of the weak formulations of these isomorphisms, this can be stated as: the gradient, $\grad$, is replaced by $L^\star$ -- the $L^2$-adjoint of $L$.} by the isomorphism $(L_wL^\star)^{-1}$ (this notation is clarified below). In general, the norm $\lV L(\cdot) \rV_{-1}$ does not control the $L^2(\Omega)$ norm; in fact, it is not even discretely (i.e., on any collection of finite element spaces) $L^2$-coercive \cite{OlsonPhDthesis}. This is associated with the difficulty in analyzing the $L^2$-convergence of the $H^{-1}$-based method in \cite{2005HdivHyperbolic,OlsonPhDthesis} (and its related $H(\div)$-conforming method). In contrast, we observe that replacing $\lV \cdot \rV_{-1}$ with the dual norm corresponding to $(L_wL^\star)^{-1}$ precisely recovers the $L^2(\Omega)$ norm. In practice, this desirable property of $(L_wL^\star)^{-1}$ is lost when the operator is approximated by a discrete version. We demonstrate that under certain conditions a discrete $L^2$-coercivity of the $(\cL\cL^\star)^{-1}$ method remains valid which is sufficient for obtaining optimal convergence rates. All methods studied in this paper converge in the $L^2(\Omega)$ norm. Since operators play such an important role in our considerations, we provide an overview of the properties of the operators of interest here.

Negative-norm least-squares methods can be viewed as particular Petrov-Galerkin finite element methods, since Petrov-Galerkin methods constitute a very wide class; see \cite{2010DPGTransport,2011DPGTestFunc} and the references therein. This paper follows a slightly different path, in a sense, more in the spirit of least-squares methods. Namely, we extend the standard $\cL\cL^\star$ method of \cite{2001FOSLLstar} either by further projections onto $\cU^h$ constituting the $\cL\cL^\star$-type methods, or by employing a related negative-norm minimization resulting in the $(\cL\cL^\star)^{-1}$ method. All methods of this paper are fundamentally based on the original $\cL\cL^\star$ minimization principle in \cite{2001FOSLLstar}. The relation to Petrov-Galerkin methods is interesting in its own right. The potential of further extending the $\cL\cL^\star$ approach using the (discontinuous) Petrov-Galerkin framework is a subject of future work.

The main contributions of this paper are summarized as follows. The novel $(\cL\cL^\star)^{-1}$ formulation is proposed and analyzed. Also, the idea of formulating a negative-norm least-squares method as a ``saddle-point problem'', to our knowledge, does not exist in the literature. A more typical approach is the one in \cite{1997NegFOSLS}, where the conjugate gradient method is directly applied to minimize the functional of interest. For practical purposes, they use a preconditioner (an approximate inverse of an operator) that effectively modifies the least-squares principle. In contrast, the approach here allows utilization of the original (unmodified) minimization principle. Note that the norm in the $(\cL\cL^\star)^{-1}$ method differs from the one in \cite{1997NegFOSLS}. Moreover, additional difficulties arise when using the conjugate gradient method for a modified least-squares principle in the context of hyperbolic PDEs; see \cref{sec:implementation}. The standard $\cL\cL^\star$ method is not new; it is formulated in \cite{2001FOSLLstar} in the context of elliptic problems. The single- and two-stage methods are simple extensions of the original $\cL\cL^\star$ approach. Although not in such a pure form, they can be seen as a part of the hybrid method in \cite{2013HybridLS}. The application of the $\cL\cL^\star$, single-, and two-stage methods to hyperbolic problems is, however, a new development. Most notably, the error analysis in \cref{sec:llstar} of the single- and two-stage methods in terms of the approximation properties of the involved finite element spaces was not previously known.

The outline of the rest of the paper is the following. Basic notions and assumptions are presented in \cref{sec:notation}. \Cref{sec:properties} contains a systematic overview of the properties of the operators of interest. In \cref{sec:LLstarinverse}, the $(\cL\cL^\star)^{-1}$ method is formulated and analyzed. \Cref{sec:llstar} is devoted to the $\cL\cL^\star$-type methods and their comparison to the $(\cL\cL^\star)^{-1}$ method.
In \cref{sec:implementation}, we comment on the implementation of the methods and the preconditioning of the respective linear systems.
The specifics of applying the methods to \eqref{eq:pde} are discussed in \cref{sec:application}. Particular numerical results are collected in \cref{sec:numerical} and the conclusion and possible future work are in the final \cref{sec:conclusions}.

\section{Notation, definitions, and assumptions}
\label{sec:notation}

Here, useful notation and definitions are presented. Also, a pair of basic assumptions is stated.

Consider a domain $\Omega \subset \bbR^d$ and a linear first-order differential operator, $L$, (i.e., it is a closed unbounded operator). The norm on $L^2(\Omega)$ is denoted by $\lV \cdot \rV$ and $\li \cdot,\cdot \ri$ is the respective inner product. The domain of $L$ is defined as
\[
\cD(L) = \set{u \in L^2(\Omega) \where Lu \in L^2(\Omega) \text{ and } Bu = 0},
\]
where $Bu = 0$ represents appropriate homogeneous boundary conditions. Note that $L$ is densely defined in the sense that $\cD(L)$ is dense in $L^2(\Omega)$. This is easy to see, since, clearly, the infinitely smooth compactly supported functions on $\Omega$ are contained in $\cD(L)$. Thus, $L^\star$, the $L^2$-adjoint of $L$, is a well-defined closed linear operator \cite{YosidaFA}. In general, the adjoint operator $L^\star$ and its domain, $\cD(L^\star)$, are defined as follows: if for $w \in L^2(\Omega)$ there exists $q \in L^2(\Omega)$ such that
\[
\li L u, w \ri = \li u, q \ri,\quad \forall u \in \cD(L),
\]
then we say that $w \in \cD(L^\star)$ and $L^\star w = q$. It is convenient to express $\cD(L^\star)$ as
\[
\cD(L^\star) = \set{w \in L^2(\Omega) \where L^\star w \in L^2(\Omega) \text{ and } B^\star w = 0},
\]
where $B^\star w = 0$ are the adjoint homogeneous boundary conditions. Moreover, it is known \cite{YosidaFA} that $L$ being densely defined and closed implies that $L^\star$ is also densely defined and $(L^\star)^\star = L$.

Assume that $L^\star$ satisfies a Poincar{\'e}-type inequality and that it is surjective. That is, for $c_P^\star > 0$,
\begin{align}
&c_P^\star \lV w \rV \le \lV L^\star w \rV, \quad \forall w \in \cD(L^\star),\label{asu:LstarPoincare}\tag{ASM 1}\\
&L^\star(\cD(L^\star)) = L^2(\Omega).\label{asu:Lstarsurjective}\tag{ASM 2}
\end{align}
The motivation behind these assumptions is that they are important for the theory in \cref{sec:properties} and they are satisfied by the problem of interest \eqref{eq:pde}. This is discussed in \cref{sec:application}.

Notice that assumption \eqref{asu:LstarPoincare} implies that $\cD(L^\star)$ is a Hilbert space with respect to $\lV \cdot \rV_{\cD(L^\star)} = \lV L^\star(\cdot) \rV$ and this norm is equivalent to the respective graph norm on $\cD(L^\star)$. That is, for $c_G^\star > 0$,
\[
c_G^\star(\lV w \rV^2 + \lV L^\star w \rV^2) \le \lV L^\star w \rV^2 \le \lV w \rV^2 + \lV L^\star w \rV^2, \quad \forall w \in \cD(L^\star).
\]

Denote the dual space of $\cD(L^\star)$ by $\cD'(L^\star)$. The associated functional norm is
\[
\lV \ell \rV_{\cD'(L^\star)} = \sup_{w\in \cD(L^\star)} \frac{\lv \ell(w) \rv}{\lV L^\star w \rV}, \quad \forall \ell \in \cD'(L^\star).
\]
To simplify notation, it is understood that $w \ne 0$ in the supremum and this convention is used throughout the paper. This leads to the following definitions.

\begin{definition}
Let $q \in L^2(\Omega)$ and consider the functional $\vartheta_q(w) = \li q,  L^\star w \ri$ for all $w \in \cD(L^\star)$. It is easy to see that $\vartheta_q \in \cD'(L^\star)$. Define the linear map $L_w\col L^2(\Omega) \to \cD'(L^\star)$ as $L_w q = \vartheta_q$ for all $q \in L^2(\Omega)$. The operator $L_w$ is the ``weak version'' of $L$, defined on the whole $L^2(\Omega)$.
\end{definition}

\begin{definition}\label{def:LLstarweak}
The linear map $(L_w L^\star)^{-1}\col \cD'(L^\star) \to \cD(L^\star)$ is defined through the solution of the weak problem
\begin{equation}\label{eq:LLstarweak}
\text{Find } z\in \cD(L^\star)\col \li L^\star z, L^\star w \ri = \ell(w), \quad \forall w \in \cD(L^\star),
\end{equation}
where $\ell \in \cD'(L^\star)$. That is, if $\hz \in \cD(L^\star)$ solves \eqref{eq:LLstarweak}, then $(L_w L^\star)^{-1} \ell = \hz$.
\end{definition}

Owing to \eqref{asu:LstarPoincare} and the Riesz theorem, \eqref{eq:LLstarweak} has a unique solution. Hence, $(L_w L^\star)^{-1}$ is well-defined. The notation is motivated by $(L_w L^\star)^{-1} = (L^\star)^{-1} L_w^{-1}$, which can be shown.

\begin{remark}\label{rem:equivasu}
Assumption \eqref{asu:Lstarsurjective} is equivalent (see \cite[Theorems 2.20 and 2.21]{BrezisFA}) to the assumption
\begin{equation}\label{asu:LPoincare}\tag{ASM 3}
c_P \lV u \rV \le \lV L u \rV, \quad \forall u \in \cD(L),
\end{equation}
for some constant $c_P > 0$. As above, \eqref{asu:LPoincare} implies that $\cD(L)$ is a Hilbert space with respect to $\lV \cdot \rV_{\cD(L)} = \lV L(\cdot) \rV$ and this norm is equivalent to the respective graph norm on $\cD(L)$. Similarly, \eqref{asu:LstarPoincare} is equivalent to the assumption that $L \col \cD(L) \to L^2(\Omega)$ is surjective, $L(\cD(L)) = L^2(\Omega)$.
\end{remark}

\section{Properties of the operators}
\label{sec:properties}

This section is devoted to an overview of the properties of the operators introduced in \cref{sec:notation}. The main idea is to characterize the $L^2(\Omega)$ norm in terms of the norm in $\cD'(L^\star)$ and to properly represent the functional norm, aiming at obtaining, in \cref{sec:LLstarinverse}, an appropriate computable approximation of the $L^2$-norm minimization \eqref{eq:L2min}. More analytical details can be found in \cite{PhDthesis}.

To aid precision and clarity, note that $L^2(\Omega)$ can be embedded into $\cD'(L^\star)$. Indeed, the embedding operator $\cE\col L^2(\Omega) \to \cD'(L^\star)$ is defined as $\cE q = \ell_q$, for all $q \in L^2(\Omega)$, where $\ell_q(w) = \li q, w\ri$ for all $w \in \cD(L^\star)$. Using \eqref{asu:LstarPoincare}, it is easy to see that $\ell_q \in \cD'(L^\star)$ and $\cE$ is a bounded linear operator, representing the continuous embedding of $L^2(\Omega)$ into $\cD'(L^\star)$.

The operator $(L_w L^\star)^{-1} \cE$ maps $L^2(\Omega)$ into $\cD(L^\star)$. Owing to \eqref{eq:LLstarweak} and the definition of $\cE$, for $q\in L^2(\Omega)$, $(L_w L^\star)^{-1} \cE q$ equals the solution of the weak problem
\begin{equation}\label{eq:LLstarweakL2}
\text{Find } z\in \cD(L^\star)\col \li L^\star z, L^\star w \ri = \li q, w \ri, \quad \forall w \in \cD(L^\star).
\end{equation}
As customary, for simplicity, we skip the embedding, $\cE$, in the notation for the operator $(L_w L^\star)^{-1} \cE$ and consider $(L_w L^\star)^{-1}\col L^2(\Omega) \to \cD(L^\star)$, if necessary, defined through the solution of the weak problem \eqref{eq:LLstarweakL2}.

The motivation behind the operator $L_w$ is that it extends $L$ (in fact, it extends $\cE L$) on $L^2(\Omega)$, in the sense that $L_w$ coincides with $\cE L$ on $\cD(L)$. This allows the general characterization of the norm in $L^2(\Omega)$ over the entire space. The result is important in the formulation of the $(\cL\cL^*)^{-1}$ method, since, as demonstrated in the next section, it essentially moves the infeasibility of \eqref{eq:L2min}, caused by the presence of the exact solution, $\hu$, to the functional norm in $\cD'(L^\star)$.

\begin{theorem}[characterization of the $L^2$ norm]\label{thm:L2char}
The operator $L_w\col L^2(\Omega) \to \cD'(L^\star)$ is a bijective isometry, showing that $\lV q \rV = \lV L_w q \rV_{\cD'(L^\star)}$, for all $q\in L^2(\Omega)$. In particular, $\lV u \rV = \lV \cE L u \rV_{\cD'(L^\star)}$, for all $u\in \cD(L)$.
\end{theorem}
\begin{proof}
Using the surjectivity of $L^\star$ in \eqref{asu:Lstarsurjective}, $L_w$ is an isometry since
\[
\lV L_w q \rV_{\cD'(L^\star)} = \sup_{w\in \cD(L^\star)} \frac{\lv \li q, L^\star w \ri \rv}{\lV L^\star w \rV} = \lV q \rV.
\]
Since $L_w$ is an isometry, it is injective. Consider arbitrary $\ell \in \cD'(L^\star)$ and let $\hz = (L_w L^\star)^{-1} \ell \in \cD(L^\star)$. By setting $\hq = L^\star \hz \in L^2(\Omega)$, it follows from \eqref{eq:LLstarweak} that $L_w \hq = \ell$. Thus, $L_w$ is surjective.
\end{proof}

It is not practical to work directly with a dual norm like $\lV\cdot\rV_{\cD'(L^\star)}$. Therefore, the operator $(L_w L^\star)^{-1}\col \cD'(L^\star) \to \cD(L^\star)$ is considered. As implied by the following lemma, it is the Riesz isomorphism between $\cD(L^\star)$ and $\cD'(L^\star)$, i.e., it is the isomorphism between a Hilbert space and its dual, mapping functionals to their representations with respect to the inner product in the Hilbert space, in accordance with the Riesz representation theorem. In essence, $(L_w L^\star)^{-1}$ is the analog of the inverse Laplace operator in $H^{-1}$-type methods.

\begin{lemma}\label{lem:LLstarinv}
The operator $(L_w L^\star)^{-1}\col \cD'(L^\star) \to \cD(L^\star)$ is a bijective isometry.
\end{lemma}
\begin{proof}
Owing to \eqref{eq:LLstarweak}, it is an isometry since
\[
\lV \ell \rV_{\cD'(L^\star)} = \sup_{w\in \cD(L^\star)} \frac{\lv \ell(w) \rv}{\lV L^\star w \rV} = \sup_{w\in \cD(L^\star)} \frac{\lv \li L^\star \hz, L^\star w \ri \rv}{\lV L^\star w \rV} = \lV L^\star \hz \rV = \lV \hz \rV_{\cD(L^\star)} = \lV (L_w L^\star)^{-1} \ell \rV_{\cD(L^\star)},
\]
where $\hz = (L_w L^\star)^{-1} \ell$. Since $(L_w L^\star)^{-1}$ is an isometry, it is injective. Owing to \eqref{asu:LstarPoincare}, \eqref{eq:LLstarweak}, and the Riesz theorem, it follows that $(L_w L^\star)^{-1}$ is surjective. Indeed, for any $z\in \cD(L^\star)$, consider $\ell^\star_z \in \cD'(L^\star)$ defined as $\ell^\star_z(w) = \li L^\star z, L^\star w \ri$ for all $w \in \cD(L^\star)$. Then, $(L_w L^\star)^{-1} \ell^\star_z = z$.
\end{proof}

\Cref{lem:LLstarinv} allows to characterize the inner product in $\cD'(L^\star)$ using the operator $(L_w L^\star)^{-1}$. This is important for the considerations in \cref{sec:LLstarinverse}, since by approximating $(L_w L^\star)^{-1}$, the $\cD'(L^\star)$ norm is approximated, thus obtaining, in view of \cref{thm:L2char}, computationally feasible approximations of the $L^2(\Omega)$ norm and the minimization \eqref{eq:L2min}. In practical finite element formulations, the $\cD'(L^\star)$ inner product characterization is needed for functions in $L^2(\Omega)$. This is the motivation behind the following result. It shows that $\li (L_w L^\star)^{-1} \cdot, \cdot \ri$ defines an inner product in $L^2(\Omega)$, which is precisely the inner product associated with $\lV \cdot \rV_{\cD'(L^\star)}$, but restricted, via the embedding $\cE$, to $L^2(\Omega)$.

\begin{theorem}\label{thm:L2SAPD_LLstarinv}
The operator $(L_w L^\star)^{-1}\col L^2(\Omega) \to \cD(L^\star)$ is self-adjoint and positive definite with respect to the $L^2(\Omega)$ inner product, and $\lV \cE q \rV_{\cD'(L^\star)}^2 = \li (L_w L^\star)^{-1} q, q \ri$, for all $q\in L^2(\Omega)$.
\end{theorem}
\begin{proof}
Given $p,q\in L^2(\Omega)$, by \eqref{eq:LLstarweakL2}, it follows
\begin{equation}\label{eq:selfadj}
\begin{split}
\li q, (L_w L^\star)^{-1} \cE p \ri &= \li L^\star (L_w L^\star)^{-1} \cE q, L^\star (L_w L^\star)^{-1} \cE p \ri\\
&= \li L^\star (L_w L^\star)^{-1} \cE p, L^\star (L_w L^\star)^{-1} \cE q \ri = \li p, (L_w L^\star)^{-1} \cE q \ri.
\end{split}
\end{equation}
Following the discussion below \eqref{eq:LLstarweakL2}, $(L_w L^\star)^{-1} \cE$ is shortly denoted as $(L_w L^\star)^{-1}\col L^2(\Omega) \to \cD(L^\star)$. Therefore, the equality in \eqref{eq:selfadj} can be written as $\li q, (L_w L^\star)^{-1} p \ri = \li p, (L_w L^\star)^{-1} q \ri$. \Cref{lem:LLstarinv} and \eqref{eq:selfadj}, for $p = q$, imply $\lV \cE q \rV_{\cD'(L^\star)}^2 = \li q, (L_w L^\star)^{-1} \cE q \ri = \li q, (L_w L^\star)^{-1} q \ri$. Finally, $\li (L_w L^\star)^{-1} q, q \ri = \lV \cE q \rV_{\cD'(L^\star)}^2 \ge 0$, where the equality holds if and only if $q = 0$.
\end{proof}

It is easy to see, using \eqref{asu:LstarPoincare} and \cref{lem:LLstarinv}, that the bilinear form $\li (L_w L^\star)^{-1} \cdot, \cdot \ri$ is continuous on $L^2(\Omega)$, reflecting that $\lV \cdot \rV_{\cD'(L^\star)}$ is weaker than $\lV \cdot \rV$ on $L^2(\Omega)$. In general, $(L_w L^\star)^{-1}$ is not necessarily $L^2$-coercive (strictly positive definite). That is, $\li (L_w L^\star)^{-1} q, q \ri \ge \alpha \lV q \rV^2$, for all $q \in L^2(\Omega)$, does not necessarily hold for any constant $\alpha > 0$. In view of \cref{thm:L2SAPD_LLstarinv}, this reflects the fact that the $L^2(\Omega)$ norm is generally strictly stronger than the $\cD'(L^\star)$ norm on $L^2(\Omega)$.

Note that \eqref{asu:LstarPoincare} provides that the operator $(L_w L^\star)^{-1}$ is well-defined, $\cD(L^\star)$ is a Hilbert space with respect to $\lV \cdot \rV_{\cD(L^\star)} = \lV L^\star(\cdot) \rV$, and is used in the proofs of \cref{lem:LLstarinv,thm:L2SAPD_LLstarinv}. The closedness of $L$ provides $(L^\star)^\star = L$, while \eqref{asu:Lstarsurjective} is used in the proof of \cref{thm:L2char}.

\section{The \texorpdfstring{$(\cL\cL^\star)^{-1}$}{(LL*)\^{}\{-1\}} method}
\label{sec:LLstarinverse}

Here, the $(\cL\cL^\star)^{-1}$ method is presented. First, it is formulated. Next, the corresponding linear algebra equations are discussed. Finally, the properties of the discrete formulation are studied.

\subsection{Motivation and formulation}

Let $\cU^h$ be a finite element space and consider \eqref{eq:opeq}. For simplicity, $\cU^h \subset \cD(L)$ in this section. The extension of the formulation to more general finite element spaces is discussed in \cref{sec:extension}. The purpose is to obtain $u^h \in \cU^h$ that approximates the exact solution of \eqref{eq:opeq} in the $L^2(\Omega)$ norm. Owing to \cref{thm:L2char,thm:L2SAPD_LLstarinv}, the minimization \eqref{eq:L2min} can be equivalently expressed as
\begin{equation}\label{eq:L2minconf}
u^h = \argmin_{v^h \in \cU^h} \lV \cE L (v^h - \hu) \rV_{\cD'(L^\star)}^2 = \argmin_{v^h \in \cU^h} \li (L_w L^\star)^{-1} (L v^h - f), L v^h - f \ri,
\end{equation}
where $\hu \in \cD(L)$ denotes the exact solution of \eqref{eq:opeq}. Owing to the symmetry in \cref{thm:L2SAPD_LLstarinv}, this leads to the weak problem
\begin{equation}\label{eq:L2minconfweak}
\text{Find } u^h \in \cU^h\col \li (L_w L^\star)^{-1} L u^h, L v^h \ri = \li (L_w L^\star)^{-1} f, L v^h \ri,\quad\forall v^h \in \cU^h.
\end{equation}

Observe that \eqref{eq:L2minconf} and \eqref{eq:L2minconfweak} are not computationally feasible, since the effect of $(L_w L^\star)^{-1}$ cannot be computed in general. Therefore, a computable discrete version of $(L_w L^\star)^{-1}$ is necessary. To this end, consider an additional (auxiliary) finite element space $\cZ^\fh \subset \cD(L^\star)$. The discrete version of $(L_w L^\star)^{-1}$ is obtained from the discrete version of \eqref{eq:LLstarweak}, as described in the following definition.

\begin{definition}
The linear map $(L_w L^\star)_\fh^{-1}\col \cD'(L^\star) \to \cZ^\fh$ is defined through the solution of the discrete weak problem, for $\ell \in \cD'(L^\star)$,
\begin{equation}\label{eq:LLstarweakhgen}
\text{Find } z^\fh\in \cZ^\fh\col \li L^\star z^\fh, L^\star w^\fh \ri = \ell(w^\fh), \quad \forall w^\fh \in \cZ^\fh.
\end{equation}
As previously, when convenient, the operator $(L_w L^\star)_\fh^{-1}\col L^2(\Omega) \to \cZ^\fh$ is considered (via the embedding $\cE$), in which case, for $q \in L^2(\Omega)$, \eqref{eq:LLstarweakhgen} takes the form
\begin{equation}\label{eq:LLstarweakh}
\text{Find } z^\fh\in \cZ^\fh\col \li L^\star z^\fh, L^\star w^\fh \ri = \li q, w^\fh \ri, \quad \forall w^\fh \in \cZ^\fh.
\end{equation}
\end{definition}

Now, \eqref{eq:L2minconf} and \eqref{eq:L2minconfweak} can be approximated by a computable method by replacing $(L_w L^\star)^{-1}$ with $(L_w L^\star)_\fh^{-1}$. This results in the following:
\begin{gather}
u^h = \argmin_{v^h \in \cU^h} \li (L_w L^\star)_\fh^{-1} (L v^h - f), L v^h - f \ri,\label{eq:minLLstarinv}\\
\text{Find } u^h \in \cU^h\col \li (L_w L^\star)_\fh^{-1} L u^h, L v^h \ri = \li (L_w L^\star)_\fh^{-1} f, L v^h \ri,\quad\forall v^h \in \cU^h,\label{eq:LLstarinv}
\end{gather}
which constitutes the discrete $(\cL\cL^\star)^{-1}$ formulation. Alternatively, \eqref{eq:LLstarweakh} and \eqref{eq:LLstarinv} can be combined into the system
\begin{equation}\label{eq:LLstarinvsystem}
\text{Find } (u^h, z^\fh)\in \cU^h \times \cZ^\fh \col
\left\{
\begin{alignedat}{4}
&\li L^\star z^\fh, L^\star w^\fh \ri + \li u^h, L^\star w^\fh \ri &&= \li f, w^\fh \ri, &&\quad \forall w^\fh \in \cZ^\fh,\\
&\li L^\star z^\fh, v^h \ri &&= 0, &&\quad\forall v^h \in \cU^h.
\end{alignedat}
\right.
\end{equation}

In summary, exchanging $(L_w L^\star)^{-1}$ for $(L_w L^\star)_\fh^{-1}$ is practically trading the minimization of the $L^2(\Omega)$ norm of the error in \eqref{eq:L2minconf} for computational feasibility. Namely, the resulting minimization problem \eqref{eq:minLLstarinv} can be solved numerically but does not necessarily provide the $L^2$-orthogonal projection of the exact solution onto $\cU^h$. In contrast, the standard $\cL\cL^\star$ method introduced in \cite{2001FOSLLstar} solves the $L^2$ minimization \eqref{eq:L2minconf} but for the special choice $\cU^h = L^\star(\cZ^\fh)$. That is, it trades the freedom of choosing a standard finite element space in the place of $\cU^h$ for computational feasibility. Moreover, the $\cL\cL^\star$ method uses the space $\cZ^\fh$ (more precisely, the space $L^\star(\cZ^\fh)$) to approximate the exact solution, $\hu$, whereas in the $(\cL\cL^\star)^{-1}$ method introduced above, $\cZ^\fh$ serves as an auxiliary space to approximate the operator $(L_w L^\star)^{-1}\col \cD'(L^\star) \to \cD(L^\star)$ by the operator $(L_w L^\star)_\fh^{-1}\col \cD'(L^\star) \to \cZ^\fh \subset \cD(L^\star)$. See \cref{sec:llstar,sec:numerical} for further and more detailed comparisons of the $(\cL\cL^\star)^{-1}$ and other $\cL\cL^\star$-type methods. The implications of approximating the minimization problem \eqref{eq:L2minconf} by \eqref{eq:minLLstarinv} are studied in \cref{ssec:LLstarinvanalys}.

\subsection{Linear algebra equations}
\label{ssec:linalg}

Here, the algebraic systems associated with \eqref{eq:LLstarinv} and \eqref{eq:LLstarinvsystem} are formulated. Let $\set{\phi^h_i}_{i=1}^N$ and $\set{\psi^\fh_i}_{i=1}^M$ be the bases for $\cU^h$ and $\cZ^\fh$, respectively. Define the matrices $\bL \in \bbR^{M\times N}$, $\bH \in \bbR^{M\times M}$, $\bMM \in \bbR^{N\times N}$ (the $L^2(\Omega)$ mass matrix on $\cU^h$), and the vector $\brff \in \bbR^M$ as
\begin{equation}\label{eq:matvecdefs}
(\bL)_{ij} = \li \phi^h_j, L^\star \psi^\fh_i \ri, \quad (\bH)_{ij} = \li L^\star \psi^\fh_j, L^\star \psi^\fh_i \ri, \quad (\bMM)_{ij} = \li \phi^h_j, \phi^h_i \ri, \quad (\brff)_i = \li f, \psi^\fh_i \ri.
\end{equation}

The functions in $\cU^h$ and $\cZ^\fh$ can be identified with their corresponding coefficient vectors with respect to the bases of the spaces. Namely, $u^h \in \cU^h$, $\bu \in \bbR^N$ and $z^\fh \in \cZ^\fh$, $\bz \in \bbR^M$ are identified with the expansions
\[
u^h = \sum_{i=1}^N (\bu)_i \phi^h_i,\quad z^\fh = \sum_{i=1}^M (\bz)_i \psi^\fh_i.
\]

Using this notation, the weak formulation \eqref{eq:LLstarinvsystem} induces the following algebraic system of equations with a symmetric block matrix $\bbA$:
\begin{equation}\label{eq:algblocksystem}
\bbA
\begin{bmatrix*}[c]
  \bz\, \\
  \bu
\end{bmatrix*}
=
\begin{bmatrix*}[l]
  \bH & \bL\, \\
  \bL^T &
\end{bmatrix*}
\begin{bmatrix*}[c]
  \bz\, \\
  \bu
\end{bmatrix*}
=
\begin{bmatrix*}[c]
  \brff\, \\
  \bzero
\end{bmatrix*}.
\end{equation}
Note that, owing to \eqref{asu:LstarPoincare}, $\bH$ is a symmetric positive definite (SPD) matrix. Eliminating $\bz$ in \eqref{eq:algblocksystem} results in the following algebraic system for the respective Schur complement:
\begin{equation}\label{eq:algsystem0}
\bL^T \bH^{-1} \bL \bu = \bL^T \bH^{-1} \brff.
\end{equation}
Denote $\bA = \bL^T \bH^{-1} \bL \in \bbR^{N\times N}$ and $\bff = \bL^T \bH^{-1} \brff \in \bbR^N$. Then \eqref{eq:algsystem0} becomes
\begin{equation}\label{eq:algsystem}
\bA \bu = \bff,
\end{equation}
which is precisely the algebraic system induced by the weak form \eqref{eq:LLstarinv}. Indeed, since the solution of \eqref{eq:LLstarweakh} (i.e., the effect of $(L_w L^\star)_\fh^{-1}$) is computed through the effect of $\bH^{-1}$, the matrix $\bA$ corresponds to the bilinear form and $\bff$ corresponds to the right-hand side in \eqref{eq:LLstarinv}. That is,
\[
(\bA)_{ij} = \li (L_w L^\star)_\fh^{-1} L \phi^h_j, L \phi^h_i \ri, \quad (\bff)_i = \li (L_w L^\star)_\fh^{-1} f, L \phi^h_i \ri.
\]
Clearly, $\bA$ is nonsingular if and only if the matrix $\bbA$ in \eqref{eq:algblocksystem} is nonsingular.

\subsection{Analysis}
\label{ssec:LLstarinvanalys}

In this subsection, the discrete $(\cL\cL^\star)^{-1}$ formulation is analyzed and studied in detail. The major result is the error estimate for the $(\cL\cL^\star)^{-1}$ method.

It is not difficult to see that the properties of $(L_w L^\star)^{-1}\col L^2(\Omega) \to \cD(L^\star)$ are to a certain extent maintained by $(L_w L^\star)_\fh^{-1}\col L^2(\Omega) \to \cZ^\fh$. Namely, the bilinear form $\li (L_w L^\star)_\fh^{-1} \cdot, \cdot \ri$ is continuous on $L^2(\Omega)$ and $(L_w L^\star)_\fh^{-1}$ is self-adjoint and positive semidefinite with respect to the $L^2(\Omega)$ inner product. Note, however, that $(L_w L^\star)_\fh^{-1}$ is not positive definite since it has a nontrivial (and infinite-dimensional) null space. This is to be expected since $(L_w L^\star)_\fh^{-1}$ maps an infinite-dimensional space to a finite-dimensional one. Indeed, from \eqref{eq:LLstarweakh}, it follows that
\[
\cN((L_w L^\star)_\fh^{-1}) = (\cZ^\fh)^\perp = \set{q \in L^2(\Omega) \where \li q, w^\fh \ri = 0 \text{ for all } w^\fh \in \cZ^\fh }
\]
is the null space of $(L_w L^\star)_\fh^{-1}$. Since $(L_w L^\star)_\fh^{-1}$ is singular, the matrix $\bA$ (or, equivalently, the matrix $\bbA$ in \eqref{eq:algblocksystem}) can be singular if the spaces $\cU^h$ and $\cZ^\fh$ are not selected carefully.

The analysis of the discrete $(\cL\cL^\star)^{-1}$ formulation and the properties of the matrix $\bA$ is fundamentally founded on the effect of replacing $(L_w L^\star)^{-1}$ with $(L_w L^\star)_\fh^{-1}$ on the characterization of the $L^2(\Omega)$ norm. \Cref{thm:L2char,thm:L2SAPD_LLstarinv} show that $(L_w L^\star)^{-1}$ together with $L_w$ exactly recover the $L^2(\Omega)$ norm on the entire space, which is related to the equality $L^\star (L_w L^\star)^{-1} L_w = I$, where $I\col L^2(\Omega) \to L^2(\Omega)$ is the identity operator. However, replacing $(L_w L^\star)^{-1}$ with $(L_w L^\star)_\fh^{-1}$ cannot fully recover the $L^2(\Omega)$ norm. The following result shows that, instead, the $L^2(\Omega)$ norm is exactly recovered only on a subspace,  $L^\star(\cZ^\fh)$, and $L^\star (L_w L^\star)^{-1}_\fh L_w$ becomes a $L^2$-orthogonal projection.

\begin{lemma}[$L^2$-orthogonal projection]\label{lem:L2proj}
Let $\Pi^\fh_\star\col L^2(\Omega) \to L^\star(\cZ^\fh)$ be the $L^2$-orthogonal projection onto $L^\star(\cZ^\fh)$, then $\Pi^\fh_\star = L^\star (L_w L^\star)_\fh^{-1} L_w$.
\end{lemma}
\begin{proof}
Consider an arbitrary $q \in L^2(\Omega)$. Notice that $\Pi^\fh_\star q \in L^\star(\cZ^\fh)$ is characterized by the weak form
\begin{equation}\label{eq:projectionweak}
\li \Pi^\fh_\star q, L^\star w^\fh \ri = \li q, L^\star w^\fh \ri,\quad \forall w^\fh \in \cZ^\fh.
\end{equation}
Let $\hz^\fh = (L_w L^\star)_\fh^{-1} L_w q \in \cZ^\fh$. The definitions of $(L_w L^\star)_\fh^{-1}$ and $L_w$ imply
\[
\li L^\star \hz^\fh, L^\star w^\fh \ri = \li q, L^\star w^\fh \ri,\quad \forall w^\fh \in \cZ^\fh.
\]
Thus, $\Pi^\fh_\star q = L^\star \hz^\fh$ and, hence, $\Pi^\fh_\star q = L^\star (L_w L^\star)_\fh^{-1} L_w q$.
\end{proof}

\Cref{lem:L2proj} is important for the coming considerations and results. Particularly, it contributes to easily establishing the following basic and useful properties of the matrix $\bA$, including an abstract characterization of its null space and a necessary condition for its nonsingularity. They assist the argumentation and motivation of the results below.

\begin{proposition}\label{prop:Aprop}
The following properties hold:
\begin{enumerate}[label=(\roman*),nolistsep]
\item\label{cor:L2proj} The matrix $\bA$ and the vector $\bff$ satisfy
\[
(\bA)_{ij} = \li \Pi^\fh_\star \phi^h_j, \phi^h_i \ri, \quad (\bff)_i = \li \Pi^\fh_\star \hu, \phi^h_i \ri,
\]
and the discrete $(\cL\cL^\star)^{-1}$ formulation in \eqref{eq:minLLstarinv} and \eqref{eq:LLstarinv} can be equivalently expressed as
\begin{gather}
u^h = \argmin_{v^h \in \cU^h} \lV \Pi^\fh_\star (v^h - \hu) \rV^2,\label{eq:minLLstarinvproj}\\
\text{Find } u^h \in \cU^h\col \li \Pi^\fh_\star u^h, v^h \ri = \li \Pi^\fh_\star \hu, v^h \ri,\quad\forall v^h \in \cU^h.\label{eq:LLstarinvproj}
\end{gather}
\item\label{cor:Aspsd} The matrix $\bA$ in \eqref{eq:algsystem} is symmetric positive semidefinite, for all choices of $\cU^h$ and $\cZ^\fh$.
\item\label{lem:Anull} By identifying the vectors in $\bbR^N$ with the functions in $\cU^h$, the null space of $\bA$ is characterized as $\cN(\bA) = \cU^h \cap \lb L^\star(\cZ^\fh) \rb^\perp$.
\item\label{cor:ALnull} The null spaces of $\bA$ and $\bL$ coincide. That is, $\cN(\bA) = \cN(\bL)$.
\item\label{cor:Asingular} If $\dim(\cU^h) > \dim(\cZ^\fh)$ (i.e., $N > M$), then $\bA$ (as well as $\bbA$) is singular.
\end{enumerate}
\end{proposition}
\begin{proof}
\vspace{-2ex}
\begin{enumerate}[label=(\roman*),nolistsep]
\item It follows from \cref{lem:L2proj}, using the equality $f = L \hu$ and that $L_w$ and $\cE L$ coincide on $\cD(L)$.
\item This is an immediate consequence of \ref{cor:L2proj}.
\item Let $v^h \in \cU^h$ be a finite element function with a coefficient vector $\bv \in \bbR^N$. By \ref{cor:L2proj}, $\bA\bv = \bzero$ holds if and only if $\Pi^\fh_\star v^h \perp \Pi^\fh_\star(\cU^h)$. The last is equivalent to $\Pi^\fh_\star v^h = 0$ and $v^h \perp L^\star(\cZ^\fh)$.
\item Let $v^h \in \cU^h$ be a finite element function with a coefficient vector $\bv \in \bbR^N$. Then
\[
(\bL \bv)_i = \li v^h, L^\star \psi^\fh_i \ri, \quad \text{for } i = 1, \dots, M,
\]
implies that $\bv \in \cN(\bL)$ if and only if $v^h \in [ L^\star(\cZ^\fh) ]^\perp$. Thus, owing to \ref{lem:Anull}, $\cN(\bA) = \cN(\bL)$.
\item Notice that \ref{lem:Anull} implies that $\bA$ is always singular if $L^\star(\cZ^\fh) \subsetneq \cU^h$. More generally, $\bA$ is guaranteed to be singular if $\cU^h$ is of higher dimension than $\cZ^\fh$. Indeed, if $N > M$, there exists $\bv \in \bbR^N \setminus \{ \bzero \}$ such that $\bv \in \cN(\bL)$. Thus, by \ref{cor:ALnull}, $\bv \in \cN(\bA)$ and, hence, $\cN(\bA) \ne \{ \bzero \}$.\qedhere
\end{enumerate}
\end{proof}

As shown in \cite{2001FOSLLstar}, discussed later in \cref{sec:llstar}, and evident from \cref{lem:L2proj}, the result of $\Pi^\fh_\star \hu$ is computable through an application of $(L_w L^\star)_\fh^{-1}$ (i.e., by solving \eqref{eq:LLstarweakhgen}). This is a feature provided by the standard $\cL\cL^\star$ method. In particular, the $\cL\cL^\star$ method of \cite{2001FOSLLstar} approximates the exact solution, $\hu$, by $\Pi^\fh_\star \hu$. This justifies why formulations like \eqref{eq:minLLstarinvproj} and \eqref{eq:LLstarinvproj} are computationally feasible. \Cref{prop:Aprop}\ref{cor:L2proj} is rather useful and interesting. It explains the effect on \eqref{eq:L2minconf} and \eqref{eq:L2minconfweak} when $(L_w L^\star)^{-1}$ is replaced by $(L_w L^\star)^{-1}_\fh$. Namely, the infeasible $L^2$-norm minimization of the error becomes a feasible, due to the standard $\cL\cL^\star$ formulation, minimization of the projection of the error. This is, generally, a semi-norm minimization that only partially represents the $L^2(\Omega)$ norm, due to the necessary discretization of the operator $(L_w L^\star)^{-1}$. Furthermore, \cref{prop:Aprop}\ref{cor:L2proj} contributes to a considerable simplification of the proofs and considerations below.

Note that \cref{prop:Aprop}\ref{cor:Asingular} establishes a necessary condition ($\dim(\cU^h) \le \dim(\cZ^\fh)$) for the invertibility of $\bA$. A sufficient condition is more delicate. \Cref{prop:Aprop}\ref{lem:Anull} suggests that the spaces $\cU^h$ and $L^\star(\cZ^\fh)$ should be ``close'' in a certain sense. This is made precise by the ``inf-sup'' condition below, which can be interpreted as a condition on the cosine of the abstract angle between the spaces $\cU^h$ and $L^\star(\cZ^\fh)$. Moreover, it implies a discrete (i.e., on $\cU^h$) $L^2$-coercivity that is a stronger result than the nonsingularity of $\bA$ and, in particular, provides information on the conditioning of $\bA$. That is, even though the $L^2(\Omega)$ norm is only partially recovered, i.e., only on $L^\star(\cZ^\fh)$, by the projection operator, the ``closeness'' of $\cU^h$ and $L^\star(\cZ^\fh)$ provided by the inf-sup condition implies a discrete ``control'' of the $L^2(\Omega)$ norm on $\cU^h$.

\begin{theorem}[inf-sup condition]\label{thm:infsup}
If there exists a constant $c_I > 0$ such that
\begin{equation}\label{eq:infsup}
\inf_{v^h \in \cU^h} \sup_{w^\fh \in \cZ^\fh} \frac{\lv \li v^h, L^\star w^\fh \ri \rv}{\lV v^h \rV \lV L^\star w^\fh \rV} \ge c_I,
\end{equation}
then the following spectral estimate holds:
\begin{equation}\label{eq:AoverM}
c_I^2\; \bv^T \bMM \bv \le \bv^T \bA \bv, \quad\forall\bv\in\bbR^N.
\end{equation}
In particular, $\bA$ (as well as $\bbA$) is nonsingular.
\end{theorem}
\begin{proof}
Consider a finite element function $v^h \in \cU^h$ and its corresponding coefficient vector $\bv \in \bbR^N$. Then, owing to \eqref{eq:infsup}, \eqref{eq:projectionweak}, and \cref{prop:Aprop}\ref{cor:L2proj}, it follows
\begin{align*}
c_I^2\; \bv^T \bMM \bv &= c_I^2 \lV v^h \rV^2 \le \lb \sup_{w^\fh \in \cZ^\fh} \frac{\lv \li v^h, L^\star w^\fh \ri \rv}{\lV L^\star w^\fh \rV} \rb^2 = \lb \sup_{w^\fh \in \cZ^\fh} \frac{\lv \li \Pi^\fh_\star v^h, L^\star w^\fh \ri \rv}{\lV L^\star w^\fh \rV} \rb^2\\
&= \lV \Pi^\fh_\star v^h \rV^2 = \li \Pi^\fh_\star v^h, \Pi^\fh_\star v^h \ri = \li \Pi^\fh_\star v^h, v^h \ri = \bv^T \bA \bv.\qedhere
\end{align*}
\end{proof}

Almost the same argument can be used to show that if $\lambda_{\min}$ is the smallest eigenvalue of the generalized eigenvalue problem $\bA \bv = \lambda \bMM \bv$, then
\[
\sqrt{\lambda_{\min}} = \inf_{v^h \in \cU^h} \sup_{w^\fh \in \cZ^\fh} \frac{\lv \li v^h, L^\star w^\fh \ri \rv}{\lV v^h \rV \lV L^\star w^\fh \rV}.
\]
Thus, \eqref{eq:AoverM} holds if and only if \eqref{eq:infsup} holds. Also, $\bA$ (as well as $\bbA$) is nonsingular if and only if
\[
\inf_{v^h \in \cU^h} \sup_{w^\fh \in \cZ^\fh} \frac{\lv \li v^h, L^\star w^\fh \ri \rv}{\lV v^h \rV \lV L^\star w^\fh \rV} > 0.
\]

The spectral inequality that is the reverse of \eqref{eq:AoverM} is easy to show, without requiring \eqref{eq:infsup}, since it is a consequence of the basic properties of the orthogonal projection. This and \eqref{eq:AoverM} are combined to obtain the spectral equivalence between $\bA$ and $\bMM$
\begin{equation}\label{eq:specequiv}
c_I^2\; \bv^T \bMM \bv \le \bv^T \bA \bv \le \bv^T \bMM \bv, \quad\forall\bv\in\bbR^N,
\end{equation}
which can be equivalently expressed in the following ways:
\begin{gather}
c_I^2 \lV v^h \rV^2 \le \li (L_w L^\star)_\fh^{-1} L v^h, L v^h \ri \le \lV v^h \rV^2, \quad\forall v^h \in \cU^h,\nn\\
c_I^2 \lV v^h \rV^2 \le \lV \Pi^\fh_\star v^h \rV^2 \le \lV v^h \rV^2, \quad\forall v^h \in \cU^h.\label{eq:discrL2equiv}
\end{gather}

As it can be expected, \eqref{eq:infsup} allows us to derive an important error estimate, which is the main result in this section. Indeed, while the operator $(L_w L^\star)_\fh^{-1}$ recovers the $L^2(\Omega)$ norm only partially, in the sense that only a projection is obtained in \cref{lem:L2proj} and not the identity operator, and it is clear from \cref{prop:Aprop}\ref{cor:L2proj} that a uniform (i.e., on the entire $L^2(\Omega)$) $L^2$-coercivity cannot hold, the discrete (on $\cU^h$) control of the $L^2(\Omega)$ norm that is provided by \eqref{eq:infsup} is sufficient for obtaining optimal convergence rates with respect to the $L^2(\Omega)$ norm. This is the content of the following abstract lemma which provides the analytical foundation for the error estimate below regarding the $(\cL\cL^*)^\mone$ method. It is a particular extension of C\'ea's lemma (see, e.g, \cite{BrennerFEM,ErnFEM}) for formulations with symmetric bilinear forms. No proof is provided since the result can be viewed as a specific adaptation of the general considerations in \cite{1971ErrBounds} and can be easily shown by a standard argument from the finite element analysis of so called ``variational crimes'' \cite[Chapter 10]{BrennerFEM}; see \cite{PhDthesis} for a proof, see also \cite{2010DPGTransport,BoffiMFE}.

\begin{lemma}\label{lem:likeCea}
Consider a real Hilbert space $\cH$ with a norm $\lV \cdot \rV_\cH$, a symmetric positive semidefinite bilinear form $a\col \cH \times \cH \to \bbR$, and a closed subspace $\cV^h \subset \cH$. Let $a$ satisfy, for some constants $\alpha,\beta > 0$,
\[
\alpha \lV \chi^h \rV_\cH \le a(\chi^h, \chi^h)^\halff, \quad a(w, w)^\halff \le \beta \lV w \rV_\cH,\quad \forall \chi^h \in \cV^h, w \in \cH.
\]
If $v^h \in \cV^h$ and $\hv \in \cH$ satisfy the ``orthogonality'' relation
\begin{equation}\label{eq:Galerkinorth}
a(v^h - \hv, \chi^h) = 0,\quad \forall \chi^h\in\cV^h,
\end{equation}
then the following (quasi-)optimal error estimate holds:
\[
\lV v^h - \hv \rV_\cH \le \lp 1 + \frac{\beta}{\alpha} \rp \inf_{\chi^h \in \cV^h} \lV \chi^h - \hv \rV_\cH.
\]
\end{lemma}

The important $L^2$-norm error estimate for the $(\cL\cL^\star)^{-1}$ method can be derived now. The argument counts on the discrete $L^2$-coercivity of the bilinear form in \eqref{eq:LLstarinvproj} given by $\eqref{eq:infsup}$, its natural $L^2$-continuity (see \eqref{eq:discrL2equiv}), and \cref{lem:likeCea} to show a (quasi-)optimal error estimate for the $(\cL\cL^\star)^{-1}$ method in the $L^2(\Omega)$ norm.

\begin{theorem}[error estimate]\label{thm:errest}
Assume that \eqref{eq:infsup} holds. If $u^h \in \cU^h$ is the approximation obtained by the $(\cL\cL^\star)^{-1}$ method (i.e., the solution to any of \eqref{eq:minLLstarinv}, \eqref{eq:LLstarinv}, \eqref{eq:LLstarinvsystem}, \eqref{eq:minLLstarinvproj}, or \eqref{eq:LLstarinvproj} obtained, e.g., by solving any of the linear systems \eqref{eq:algblocksystem} or \eqref{eq:algsystem}) and $\hu$ is the exact solution of \eqref{eq:opeq}, then
\[
\lV u^h - \hu \rV \le \lp 1 + \frac{1}{c_I} \rp \inf_{v^h \in \cU^h} \lV v^h - \hu \rV.
\]
\end{theorem}
\begin{proof}
In view of \eqref{eq:LLstarinv} and \eqref{eq:LLstarinvproj}, the bilinear forms of interest here are $\li \Pi^\fh_\star \cdot, \cdot \ri$ and $\li (L_w L^\star)_\fh^{-1} L \cdot, L \cdot \ri$. Owing to \cref{lem:L2proj}, they coincide when they are both defined, i.e., on $\cD(L)$. However, the bilinear form $\li \Pi^\fh_\star \cdot, \cdot \ri$ is clearly well-defined on $L^2(\Omega)\times L^2(\Omega)$ and it is the one that is useful for this proof. Further information on extending the $(\cL\cL^\star)^{-1}$ formulation is provided in \cref{sec:extension}.

First, since $\Pi^\fh_\star$ is an orthogonal projection, it holds that
\[
\li \Pi^\fh_\star q , q \ri^\halff \le \lV q \rV, \quad \forall q \in L^2(\Omega),
\]
i.e., $\li \Pi^\fh_\star \cdot, \cdot \ri$ is continuous on $L^2(\Omega)$. The left inequality in \eqref{eq:discrL2equiv} can be written as
\[
c_I \lV v^h \rV \le \li \Pi^\fh_\star v^h, v^h \ri^\halff,\quad\forall v^h\in\cU^h,
\]
which shows that $\li \Pi^\fh_\star \cdot, \cdot \ri$ is $L^2$-coercive on the discrete space $\cU^h$. Next, \eqref{eq:LLstarinvproj} implies the orthogonality property
\[
\li \Pi^\fh_\star (u^h - \hu), v^h \ri = 0,\quad\forall v^h \in \cU^h.
\]
Thus, the error estimate follows from \cref{lem:likeCea}.
\end{proof}

\begin{remark}
Notice that the argument in \cref{thm:errest} only needs \eqref{eq:infsup} and $\cU^h \subset L^2(\Omega)$. No other particular assumptions on $\cU^h$ are necessary as long as a general $(\cL\cL^\star)^{-1}$ formulation like \eqref{eq:minLLstarinvproj} and \eqref{eq:LLstarinvproj} is used; see \cref{sec:extension}.
\end{remark}

In general, all observations above also hold when $c_I$ depends on the mesh parameter, $h$, instead of being a constant. In such a case, according to the estimate in \cref{thm:errest}, an $h$-dependence of $c_I$ takes away from the convergence order that is implied by the approximation properties of $\cU^h$. Also, this would affect the spectral equivalence estimate \eqref{eq:specequiv}.

Obtaining inf-sup conditions of the form \eqref{eq:infsup} for common finite element spaces is nontrivial. However, for the special choice of $\cU^h = L^\star(\cZ^\fh)$, it is easy to see that \eqref{eq:infsup} holds with $c_I = 1$. In this case, $\bA = \bMM$ and \eqref{eq:minLLstarinvproj} reduces to $u^h = \Pi^\fh_\star \hu$. That is, the $(\cL\cL^\star)^{-1}$ method coincides with the standard $\cL\cL^\star$ method when $\cU^h = L^\star(\cZ^\fh)$. See \cref{sec:llstar} for a further discussion on the relation of the $(\cL\cL^\star)^{-1}$ method to other $\cL\cL^\star$-type methods. In general, it is reasonable to expect that, for any fixed $\cU^h$ (i.e., $h$ is fixed), the corresponding approximation of $(L_w L^\star)^{-1}$ by $(L_w L^\star)_\fh^{-1}$ becomes better as $\fh \to 0$, in the sense that the representation of the $L^2(\Omega)$ norm on $\cU^h$ improves. This is demonstrated below by showing, under mild assumptions on the approximation properties of $L^\star(\cZ^\fh)$, that $c_I \to 1$ in \eqref{eq:infsup} as $\fh \to 0$ and the $(\cL\cL^*)^\mone$ solution approaches the $L^2$-orthogonal projection of $\hu$ onto $\cU^h$. That is, as $\fh \to 0$, the abstract angle between the spaces $\cU^h$ and $L^\star(\cZ^\fh)$ vanishes and the computational representation of the $L^2(\Omega)$ norm on $\cU^h$ becomes closer to being exact, since it is exact on $L^\star(\cZ^\fh)$. Furthermore, it is shown, under stronger assumptions on the approximation properties of $L^\star(\cZ^\fh)$, that \eqref{eq:infsup} can be maintained uniformly with $c_I$ arbitrarily close to $1$ by taking the ratio $h/\fh$ sufficiently large and keeping it fixed. This is a very basic study of how suitable approximation properties can provide inf-sup stability by appropriately selecting the configuration of spaces. These considerations need the following proposition. It shows that the inf-sup condition \eqref{eq:infsup} can be equivalently expressed as a ``sup-inf'' condition. This can be interpreted as a condition on the sine of the abstract angle between the spaces $\cU^h$ and $L^\star(\cZ^\fh)$.

\begin{proposition}\label{prop:supinf}
The inf-sup condition \eqref{eq:infsup} is equivalent to
\begin{equation}\label{eq:supinf}
\sup_{v^h \in \cU^h}\frac{\lV v^h - \Pi^\fh_\star v^h \rV}{\lV v^h \rV} \le \sqrt{1 - c_I^2}.
\end{equation}
\end{proposition}
\begin{proof}
Using \eqref{eq:projectionweak} and the simple equality $\lV q \rV^2 = \lV \Pi^\fh_\star q \rV^2 + \lV q - \Pi^\fh_\star q \rV^2$, for all $q\in L^2(\Omega)$, the equivalence follows from
\begin{align*}
&\inf_{v^h \in \cU^h} \lb \sup_{w^\fh \in \cZ^\fh} \frac{\lv \li v^h, L^\star w^\fh \ri \rv}{\lV v^h \rV \lV L^\star w^\fh \rV}\rb^2 = \inf_{v^h \in \cU^h} \frac{\lV \Pi^\fh_\star v^h \rV^2}{\lV v^h \rV^2} = \inf_{v^h \in \cU^h} \frac{\lV v^h \rV^2 - \lV v^h - \Pi^\fh_\star v^h \rV^2}{\lV v^h \rV^2}\\
&= \inf_{v^h \in \cU^h} \lb 1 - \frac{\lV v^h - \Pi^\fh_\star v^h \rV^2}{\lV v^h \rV^2} \rb = 1 - \sup_{v^h \in \cU^h} \frac{\lV v^h - \Pi^\fh_\star v^h \rV^2}{\lV v^h \rV^2}.\qedhere
\end{align*}
\end{proof}

There are no explicit requirements on the approximation properties of $\cZ^\fh$, as long as \eqref{eq:infsup} holds. However, \cref{prop:supinf} suggests that the approximation properties of $L^\star(\cZ^\fh)$ may not be fully neglected. In fact, if $L^\star(\cZ^\fh)$ possesses such properties, the $(\cL\cL^\star)^{-1}$ method can always be made stable (in the sense that \eqref{eq:infsup} can be enforced) as long as $\fh$ is taken sufficiently small for fixed $\cU^h$. Indeed, let $h$ (i.e., $\cU^h$) be \emph{fixed} and assume that $L^\star(\cZ^\fh)$ satisfies an approximation bound like
\begin{equation}\label{eq:approx}
\lV \Pi^\fh_\star v^h - v^h \rV \le C_{v^h} \fh^\gamma,\quad\forall v^h \in \cU^h,
\end{equation}
for $\gamma > 0$ and a constant $C_{v^h} > 0$ that generally depends on some Sobolev-type norm of $v^h$. Then, one can show, for any $v^h \in \cU^h$, that
\[
\lV \Pi^\fh_\star v^h - v^h \rV \le C_h \fh^\gamma \lV v^h \rV,
\]
where the constant $C_h > 0$ can depend on the space $\cU^h$. Therefore, \eqref{eq:supinf} becomes arbitrary small, when $\fh$ is sufficiently close to zero. More precisely, $(1 - c_I^2)^\halff = \cO(\fh^\gamma)$ and $1 - c_I = \cO(\fh^{2\gamma})$, using $c_I\in[0,1]$ and the trivial $1-c_I^2 = (1-c_I)(1+c_I)$. That is, the inf-sup condition can be enforced with a constant $c_I$ arbitrary close to $1$, as long as $\fh$ is taken sufficiently small, for fixed $h$.

Intuitively, this means that, as $\fh \to 0$, $(L_w L^\star)_\fh^{-1}$ approaches $(L_w L^\star)^{-1}$, the discrete $(\cL\cL^\star)^{-1}$ formulation \eqref{eq:minLLstarinv} approaches the $L^2$-norm minimization \eqref{eq:L2minconf}, and the $(\cL\cL^\star)^{-1}$ approximation, $u^h$, approaches the $L^2$-orthogonal projection of $\hu$ onto $\cU^h$. Indeed, consider the vector $\bhff \in \bbR^N$:
\[
(\bhff)_i = \li \hu, \phi^h_i \ri.
\]
Then, the $L^2$-norm minimization \eqref{eq:L2minconf} induces the linear system
\begin{equation}\label{eq:projalgsystem}
\bMM \bu_p = \bhff,
\end{equation}
where $u^h_p$ denotes the $L^2$-orthogonal projection of $\hu$ onto $\cU^h$ and $\bu_p \in \bbR^N$ is its respective coefficient vector. One can show that
\[
\lv \bA - \bMM \rv = \cO(\fh^\gamma), \quad \lv \bff - \bhff \rv = \cO(\fh^\gamma),
\]
for any vector and its respective matrix norms $\lv \cdot \rv$. Thus, the $(\cL\cL^\star)^{-1}$ linear system \eqref{eq:algsystem} approaches the $L^2$-orthogonal projection linear system \eqref{eq:projalgsystem}, for fixed $h$, as $\fh \to 0$. A well-known perturbation result from linear algebra (see, e.g., \cite[Theorem 2.3.8]{WatkinsLA}) implies that $\bu$ also approaches $\bu_p$:
\[
\lv \bu - \bu_p \rv = \cO(\fh^\gamma) \quad\text{and}\quad \lV u^h - u^h_p \rV = \cO(\fh^\gamma).
\]
Recall that, here, $u^h \in \cU^h$ denotes the approximation obtained by the $(\cL\cL^\star)^{-1}$ method, $\bu \in \bbR^N$ is its respective coefficient vector, and $h$ is fixed as $\fh$ approaches zero.

The above argument does not exclude the possibility that, in general, the ratio $h/\fh$ may potentially need to grow to maintain \eqref{eq:infsup} as $h\to 0$. However, assume $\cU^h$ is an $H^1$ (Lagrangian) finite element space on a quasi-uniform mesh, $\Omega$ is a polyhedral (or polygonal) domain, and it holds
\[
\lV \Pi^\fh_\star v^h - v^h \rV \le C \fh \lV v^h \rV_1,\quad\forall v^h \in \cU^h,
\]
where $\lV \cdot \rV_1$ is the norm on $H^1(\Omega)$ and the constant $C>0$ does not depend on $h$, $\fh$, or $v^h$. That is, at least to a certain extent, the approximation properties of $L^\star(\cZ^\fh)$ are on par with those of $\cU^h$. Let $\fh = h/\tau$ for some constant $\tau \ge 1$. Then, using an inverse inequality \cite[Theorem 4.5.11]{BrennerFEM}, we obtain
\[
\frac{\lV \Pi^\fh_\star v^h - v^h \rV}{\lV v^h \rV} \le \frac{C}{\tau}.
\]
Thus, if $\tau$ is sufficiently large (i.e., $\fh$ is sufficiently small relative to $h$), then \eqref{eq:supinf} (and \eqref{eq:infsup}) can be enforced with $c_I$ arbitrary close to $1$ and the inf-sup condition is maintained as $h \to 0$ by keeping the ratio $h/\fh = \tau$ fixed. Similar to above, observe that $(1 - c_I^2)^\halff = \cO(\tau^{-1})$ and $1 - c_I = \cO(\tau^{-2})$.

In the discrete $(\cL\cL^\star)^{-1}$ formulation \eqref{eq:LLstarinv}, $(L_w L^\star)^{-1}$ is replaced by $(L_w L^\star)_\fh^{-1}$ (i.e., $\cD(L^\star)$ is replaced by $\cZ^\fh$) leading to the loss of the $L^2$-orthogonal projection property of \eqref{eq:L2minconfweak}. However, \cref{thm:errest} shows that if $\cZ^\fh$ is appropriately chosen in relation to $\cU^h$, so that \eqref{eq:infsup} would hold, then the approximation $(L_w L^\star)_\fh^{-1}$ of $(L_w L^\star)^{-1}$ is of sufficient quality to guarantee (quasi-)optimal $L^2$-norm approximations on $\cU^h$ of the exact solution. The above considerations show that under mild assumptions the $(\cL\cL^\star)^{-1}$ method can be made stable (i.e., \eqref{eq:infsup} can be enforced) and under stronger assumptions this can be achieved with a fixed ratio $h/\fh$. Deriving inf-sup conditions like \eqref{eq:infsup} for spaces $\cU^h$, $\cZ^\fh$ and operators $L$, $L^\star$ of interest is currently an open question, especially for $h/\fh$ being fixed and small so that the method is computationally efficient. It is not clear if this can be achieved with common finite element spaces serving as $\cZ^\fh$ or special (ad-hoc) spaces are needed to guarantee \eqref{eq:infsup}. In \cref{sec:numerical}, we investigate numerically the behavior of the method on model problems, using common finite element spaces as $\cZ^\fh$ in which case \eqref{eq:infsup} may not hold uniformly.

\section{Other \texorpdfstring{$\cL\cL^\star$}{LL*}-type methods}
\label{sec:llstar}

This section is devoted to more standard $\cL\cL^\star$-type approaches. All methods here and the $(\cL\cL^\star)^{-1}$ method of the previous section are related as they are founded upon the original $\cL\cL^\star$ method introduced in \cite{2001FOSLLstar}. Here, we consider all formulations on common terms to aid the comparison between them. They are further compared numerically in \cref{sec:numerical}. Here, for simplicity, $\cU^h \subset \cD(L)$, while extensions to more general finite element spaces are discussed in \cref{sec:extension}.

First, consider the (standard) $\cL\cL^\star$ formulation of \cite{2001FOSLLstar}:
\begin{equation}\label{eq:stdLLstarmin}
z^\fh_\star = \argmin_{w^\fh \in \cZ^\fh} \lV L^\star w^\fh - \hu \rV^2.
\end{equation}
The resulting $\cL\cL^\star$ approximation is $u^\fh_\star = L^\star z^\fh_\star \in L^\star(\cZ^\fh)$. The weak form corresponding to \eqref{eq:stdLLstarmin} is
\[
\text{Find } z^\fh\in \cZ^\fh\col \li L^\star z^\fh, L^\star w^\fh \ri = \li f, w^\fh \ri, \quad \forall w^\fh \in \cZ^\fh.
\]
That is, the weak form is \eqref{eq:LLstarweakh} with $q = f$, i.e., $z^\fh_\star = (L_w L^\star)_\fh^{-1} f$ and $u^\fh_\star = \Pi^\fh_\star \hu$. The method provides the best $L^2$-norm approximation of $\hu$ in $L^\star(\cZ^\fh)$. The quality of $u^\fh_\star$ depends on the approximation properties of $L^\star(\cZ^\fh)$. Using the notation in \eqref{eq:matvecdefs}, \eqref{eq:stdLLstarmin} induces the linear system of equations
\begin{equation}\label{eq:algstdLLstar}
\bH \bz_\star = \brff.
\end{equation}

To obtain an approximation on $\cU^h$, the $\cL\cL^\star$ solution, $u^\fh_\star$, can be further projected onto $\cU^h$:
\begin{equation}\label{eq:TSmin}
u^h_{ts} = \argmin_{v^h \in \cU^h} \lV v^h - L^\star z^\fh_\star \rV^2.
\end{equation}
Computationally, this requires solving a linear system with the mass matrix $\bMM$. The minimizations \eqref{eq:stdLLstarmin} and \eqref{eq:TSmin} constitute the \emph{``two-stage method''}. Alternatively, the minimizations in \eqref{eq:stdLLstarmin} and \eqref{eq:TSmin} can be combined resulting in the \emph{``single-stage method''}:
\begin{equation}\label{eq:OSmin}
(u^h_{ss}, z^\fh_\bullet) = \argmin_{(v^h, w^\fh) \in \cU^h \times \cZ^\fh} \lb \omega \lV L^\star w^\fh - \hu \rV^2 + \lV v^h - L^\star w^\fh \rV^2 \rb,
\end{equation}
for a given constant weight $\omega > 0$. Note that $L^\star z^\fh_\bullet  \in L^\star(\cZ^\fh)$ also approximates $\hu$ but it is generally inferior, as a $L^2$-norm approximation, to the standard $\cL\cL^\star$ solution, $u^\fh_\star$, since $u^\fh_\star$ is the best $L^2$-norm approximation of $\hu$ in $L^\star(\cZ^\fh)$. Also, the purpose here is to obtain approximations in $\cU^h$. Therefore, we concentrate on $u^h_{ss}$. Formulation \eqref{eq:OSmin} resembles the ``hybrid method'' introduced in \cite{2013HybridLS} with the difference that the first-order system least-squares (FOSLS) term is not present in \eqref{eq:OSmin}.

As in \cref{ssec:linalg}, \eqref{eq:TSmin} and \eqref{eq:OSmin} induce the respective block linear systems (cf., \eqref{eq:algblocksystem})
\begin{align}
\begin{bmatrix*}[l]
  \phantom{-}\bH & \\
  -\bL^T & \bMM\,
\end{bmatrix*}
\begin{bmatrix*}[c]
  \bz_\star\, \\
  \bu_{ts}
\end{bmatrix*}
&=
\begin{bmatrix*}[c]
  \brff\, \\
  \bzero
\end{bmatrix*},\nn\\
\bbA_{ss}
\begin{bmatrix*}[c]
  \bz_\bullet\, \\
  \bu_{ss}
\end{bmatrix*}
=
\begin{bmatrix*}[l]
  (\omega + 1)\bH & \,-\bL\, \\
  \phantom{(\omega + )}\,\mathopen{}-\bL^T & \phantom{-}\bMM\,
\end{bmatrix*}
\begin{bmatrix*}[c]
  \bz_\bullet\, \\
  \bu_{ss}
\end{bmatrix*}
&=
\begin{bmatrix*}[c]
  \omega \brff\, \\
  \bzero
\end{bmatrix*}.\label{eq:OSalgblocksystem}
\end{align}
Similar to \eqref{eq:algsystem0}, $\bz_\star$ and $\bz_\bullet$ can be eliminated, resulting in problems involving only $\bu_{ts}$ and $\bu_{ss}$. Namely, using the notation $\bff = \bL^T \bH^{-1} \brff \in \bbR^N$ introduced above \eqref{eq:algsystem}, the algebraic systems for the respective Schur complements corresponding to the methods in this paper are the following:
\begin{alignat}{3}
\bA &\bu_{inv} &&= \bff\quad &&\text{($(\cL\cL^\star)^{-1}$ method)},\label{eq:LLstarinvalgsystem}\\
\bMM &\bu_{ts} &&= \bff\quad &&\text{(two-stage method)},\label{eq:TSalgsystem}\\
[(\omega + 1) \bMM - \bA] &\bu_{ss} &&= \omega \bff\quad &&\text{(single-stage method)}.\label{eq:OSalgsystem}
\end{alignat}

\Cref{prop:Aprop}\ref{cor:L2proj} demonstrates that the algebraic system \eqref{eq:LLstarinvalgsystem} precisely corresponds to the least-squares problem \eqref{eq:minLLstarinvproj} that minimizes the $L^\star(\cZ^\fh)$ component of the error. It is possible to obtain similar minimization problems that characterize the solutions to \eqref{eq:TSalgsystem} and \eqref{eq:OSalgsystem} in relation to the exact solution, $\hu$, aiding the comparison between the methods. Namely, the algebraic systems \eqref{eq:LLstarinvalgsystem}, \eqref{eq:TSalgsystem}, and \eqref{eq:OSalgsystem} are associated with the following respective least-squares problems:
\begin{align}
u^h_{inv} &= \argmin_{v^h \in \cU^h} \lV \Pi^\fh_\star (v^h - \hu) \rV^2,\label{eq:minLLstarinvproj1}\\
u^h_{ts} &= \argmin_{v^h \in \cU^h} \lb \lV \Pi^\fh_\star (v^h - \hu) \rV^2 + \lV v^h - \Pi^\fh_\star v^h \rV^2 \rb,\label{eq:minTSproj}\\
u^h_{ss} &= \argmin_{v^h \in \cU^h} \lb \lV \Pi^\fh_\star (v^h - \hu) \rV^2 +  \frac{\omega+1}{\omega} \lV v^h - \Pi^\fh_\star v^h \rV^2 \rb.\label{eq:minOSproj}
\end{align}
It is not difficult to derive \eqref{eq:minTSproj} from \eqref{eq:TSmin}. By observing that $\lV v^h - \Pi^\fh_\star \hu \rV^2 = \lV \Pi^\fh_\star (v^h - \hu) \rV^2 + \lV (I - \Pi^\fh_\star) v^h \rV^2$, for any $v^h \in \cU^h$, it is easy to see that the weak form corresponding to \eqref{eq:minTSproj} induces the linear system \eqref{eq:TSalgsystem}. Similarly, \eqref{eq:minOSproj} can be derived from \eqref{eq:OSmin} but it is more challenging. Nevertheless, it is easy to verify that \eqref{eq:OSalgsystem} can be associated with the weak formulation
\[
\text{Find } u^h \in \cU^h\col \frac{\omega+1}{\omega} \li (I - \Pi^\fh_\star) u^h, v^h \ri + \li \Pi^\fh_\star u^h, v^h \ri = \li \Pi^\fh_\star \hu, v^h \ri,\quad\forall v^h \in \cU^h,
\]
which, clearly, corresponds to the minimization \eqref{eq:minOSproj}; see the proof of \cref{thm:OSTSerrest} below.

In comparison, the $L^2$-orthogonal projection is defined as
\[
u^h_{p} = \argmin_{v^h \in \cU^h} \lV v^h - \hu \rV^2 = \argmin_{v^h \in \cU^h} \lb \lV \Pi^\fh_\star (v^h - \hu) \rV^2 + \lV (I - \Pi^\fh_\star) (v^h - \hu) \rV^2 \rb,
\]
but this formulation is generally infeasible because of the second term. Indeed, $(I - \Pi^\fh_\star)\hu$ is generally not computationally obtainable, whereas $\Pi^\fh_\star \hu$ is available via the $\cL\cL^\star$ method \eqref{eq:stdLLstarmin}. Thus, all three methods \eqref{eq:minLLstarinvproj1}, \eqref{eq:minTSproj}, and \eqref{eq:minOSproj} trade the $L^2$-orthogonal projection for computational feasibility. The difference is that \eqref{eq:minLLstarinvproj1} drops the term for which there is no information, while \eqref{eq:minTSproj} and \eqref{eq:minOSproj} replace it with ``regularization'' terms for the size of $(I - \Pi^\fh_\star) v^h$, i.e., they only drop $(I - \Pi^\fh_\star)\hu$. Note that the second terms in \eqref{eq:minTSproj} and \eqref{eq:minOSproj} cannot be expected to contribute to the quality of approximation, since they do not contain information on $\hu$. However, those terms ``stabilize'' the methods and the matrices in \eqref{eq:TSalgsystem} and \eqref{eq:OSalgsystem} are always SPD (hence, nonsingular).

\begin{remark}
Observe that the formulation in \eqref{eq:minOSproj} approaches the one in \eqref{eq:minTSproj} as $\omega \to \infty$. In fact, the two-stage method can be viewed as an extreme case of the single-stage method, when $\omega = \infty$.
\end{remark}

Next, error estimates for the single- and two-stage methods are derived.

\begin{theorem}[error estimate]\label{thm:OSTSerrest}
The following error estimate holds:
\[
\lV u^h_\diamond - \hu \rV \le s \inf_{v^h \in \cU^h} \lV v^h - \hu \rV + (s+1)\inf_{w^\fh \in \cZ^\fh} \lV L^\star w^\fh - \hu \rV,
\]
where $u^h_\diamond = \{u^h_{ss} \text{ or } u^h_{ts}\}$, $s = \{(\omega+1)/\omega \text{ or } 1\}$ for the single- and two-stage methods, respectively.
\end{theorem}
\begin{proof}
The weak form associated with \eqref{eq:minOSproj} (or \eqref{eq:minTSproj}), i.e., the one that induces \eqref{eq:OSalgsystem} (or \eqref{eq:TSalgsystem}), is
\[
\text{Find } u^h \in \cU^h\col a_s(u^h, v^h) = \li \Pi^\fh_\star \hu, v^h \ri,\quad\forall v^h \in \cU^h,
\]
where the symmetric bilinear form $a_s\col L^2(\Omega)\times L^2(\Omega) \to \bbR$ is defined as
\[
a_s(p, q) = s \li (I - \Pi^\fh_\star) p, q \ri + \li \Pi^\fh_\star p, q \ri,\quad\forall p,q\in L^2(\Omega).
\]
Clearly, $a_s$ is $L^2$-equivalent, i.e.,
\[
\lV q \rV^2 \le a_s(q, q) \le s \lV q \rV^2,\quad\forall q\in L^2(\Omega).
\]
Also, the following ``orthogonality'' property holds:
\[
a_s(u^h_\diamond - \Pi^\fh_\star \hu, v^h) = 0,\quad\forall v^h \in \cU^h.
\]
Thus, C\'ea's lemma implies
\[
\lV u^h_\diamond - \Pi^\fh_\star \hu \rV \le s \lV v^h - \Pi^\fh_\star \hu \rV, \quad\forall v^h \in \cU^h.
\]
Using this, for any $v^h \in \cU^h$,
\begin{align*}
\lV u^h_\diamond - \hu \rV &\le \lV u^h_\diamond - \Pi^\fh_\star \hu \rV + \lV \Pi^\fh_\star \hu - \hu \rV \le s \lV v^h - \Pi^\fh_\star \hu \rV + \lV \Pi^\fh_\star \hu - \hu \rV\\
&\le s \lV v^h - \hu \rV + (s + 1)\lV \Pi^\fh_\star \hu - \hu \rV.\qedhere
\end{align*}
\end{proof}

\begin{remark}
Notice that the $(\cL\cL^\star)^{-1}$ method is the only one of the three that possesses an ``orthogonality'' property like \eqref{eq:Galerkinorth} with respect to the exact solution, $\hu$, whereas \eqref{eq:minTSproj} and \eqref{eq:minOSproj} satisfy such a property for the projection $\Pi^\fh_\star \hu$. Also, the $(\cL\cL^\star)^{-1}$ method is the only one that does not have a uniform $L^2$-coercivity and depends on \eqref{eq:infsup} to satisfy a discrete (i.e., on $\cU^h$) $L^2$-coercivity.
\end{remark}

\Cref{thm:OSTSerrest} suggests that the quality of the solutions in $\cU^h$ obtained by the single- and two-stage methods can depend not only on the approximation properties of $\cU^h$ but also on those of $L^\star(\cZ^\fh)$. In view of \eqref{eq:TSmin} and \eqref{eq:OSmin}, this can be expected since $L^\star(\cZ^\fh)$ is the only ``connection'' between the resulting solutions in $\cU^h$ and $\hu$. According to \cref{thm:OSTSerrest}, optimal rates of convergence are obtainable when the approximation properties of $L^\star(\cZ^\fh)$ are not worse than those of $\cU^h$ and an optimal setting would be if they are on par. In particular, when $\cU^h \subset L^\star(\cZ^\fh)$, then all three methods coincide ($\bA = \bMM$, \eqref{eq:infsup} holds with $c_I = 1$, and the systems \eqref{eq:LLstarinvalgsystem}, \eqref{eq:TSalgsystem}, \eqref{eq:OSalgsystem}, and \eqref{eq:projalgsystem} coincide), and they provide the $L^2$-orthogonal projection of $\hu$ onto $\cU^h$, i.e., $u^h_{inv} = u^h_{ts} = u^h_{ss} = u^h_{p}$. In general, for fixed $h$ and assuming that the property \eqref{eq:approx} holds, a similar argument to the one following \cref{prop:supinf} shows that the linear systems \eqref{eq:LLstarinvalgsystem}, \eqref{eq:TSalgsystem}, and \eqref{eq:OSalgsystem} approach \eqref{eq:projalgsystem} as $\fh \to 0$ and $\lV u^h_\diamond - u^h_p \rV = \cO(\fh^\gamma)$, where $u^h_\diamond = \set{u^h_{inv}, u^h_{ts}, \text{ or } u^h_{ss}}$. In the case $u^h_\diamond = u^h_{ss}$, the constant in the $\cO$-notation depends on $1/\omega$. That is, for fixed $h$, the three approaches converge to the same method as $\fh \to 0$, which is the $L^2$-orthogonal projection \eqref{eq:L2minconf}.

\section{Implementation and preconditioning}
\label{sec:implementation}

In this section, the implementation and preconditioning of the linear systems introduced in the previous sections is discussed. Particularly, we consider Krylov methods with block preconditioners.

The $(\cL\cL^\star)^{-1}$ method can be implemented similarly to the $H^{-1}$ method in \cite{1997NegFOSLS}. In view of \eqref{eq:specequiv}, the conjugate gradient method (CG) is potentially (depending on \eqref{eq:infsup}) an adequate choice for solving \eqref{eq:algsystem}. Obtaining a matrix-vector product with $\bA$ on each CG iteration requires computing the effect of $(L_w L^\star)_\fh^{-1}$, i.e., numerically inverting $\bH$. As in \cite{1997NegFOSLS}, $\bH^{-1}$ can be replaced by a SPD preconditioner $\bB^{-1}$. This is equivalent to replacing $(L_w L^\star)_\fh^{-1}\col \cD'(L^\star) \to \cZ^\fh$ with a respective operator $B_\fh^{-1}\col \cD'(L^\star) \to \cZ^\fh$. It results in \eqref{eq:minLLstarinv} being replaced by the modified minimization
\begin{equation}\label{eq:modmin}
\tu^h_{inv} = \argmin_{v^h \in \cU^h} \li B_\fh^{-1} (L v^h - f), L v^h - f \ri.
\end{equation}
More precisely, for $\ell \in \cD'(L^\star)$, $\tz^\fh = B_\fh^{-1} \ell$ with coefficients $\btz \in \bbR^M$ is defined as $\btz = \bB^{-1} \bell$, with $\bell \in \bbR^M$, $(\bell)_i = \ell(\psi^\fh_i)$, whereas, for $z^\fh = (L_w L^\star)_\fh^{-1} \ell$ with coefficients $\bz \in \bbR^M$, it holds $\bz = \bH^{-1} \bell$.

The weak form, associated with the minimization problem \eqref{eq:modmin}, induces the linear system
\[
\btA \btu_{inv} = \btff,
\]
where $\btA \in \bbR^{N\times N}$, $\btff \in \bbR^N$, and $(\btA)_{ij} = \li B_\fh^{-1} L \phi^h_j, L \phi^h_i \ri$, $(\btff)_i = \li B_\fh^{-1} f, L \phi^h_i \ri$.

The matrix $\btA$ (as well as $\bA$) is generally dense and it is unpractical to explicitly assemble it. However, Krylov methods can clearly be used in a matrix-free way. Matrix-vector products can be computed without assembling $\btA$. Indeed, similar to \eqref{eq:algsystem0}, $\btA \bv = \bL^T \bB^{-1} \bL \bv$, for $\bv \in \bbR^N$. Thus, computing $\btA \bv$ requires a single application of $\bB^{-1}$ (i.e, of $B_\fh^{-1}$) and matrix-vector products with $\bL$ and $\bL^T$, which can be efficiently assembled. Similarly, $\btff$ can be computed.

If $\bB$ is spectrally equivalent to $\bH$, it holds, for some constants $c_s, C_s > 0$,
\[
c_s\, \bz^T \bH^{-1} \bz \le \bz^T \bB^{-1} \bz \le C_s\, \bz^T \bH^{-1} \bz,\quad\forall \bz \in \bbR^M.
\]
Then, similar to \cref{thm:errest}, the following error estimate can be shown:
\[
\lV \tu^h_{inv} - \hu \rV \le \lp 1 + \frac{\sqrt{C_s}}{\sqrt{c_s} c_I} \rp \inf_{v^h \in \cU^h} \lV v^h - \hu \rV.
\]
That is, the modified minimization \eqref{eq:modmin} maintains the properties of the original $(\cL\cL^\star)^{-1}$ method \eqref{eq:minLLstarinv} when $\bB$ is spectrally equivalent to $\bH$.

Obtaining spectrally equivalent preconditioners of $\bH$ for hyperbolic $L^\star$ is quite challenging. In the above approach, the quality of the preconditioner can affect not only the solver but also the minimization formulation and the quality of the approximation by $\tu^h_{inv}$. Therefore, we propose a different path here, using the same tools (the preconditioner $\bB$ and Krylov solvers) and solving the block system \eqref{eq:algblocksystem} directly, thus maintaining the original $(\cL\cL^\star)^{-1}$ principle \eqref{eq:minLLstarinv}.

Based on well-known block factorizations of $2\times 2$ block matrices, we obtain the following symmetric block preconditioner of the matrix $\bbA$ in \eqref{eq:algblocksystem} (see also \cite{2011PrecondPDESystems}):
\begin{equation}\label{eq:Binv}
\bbB_{inv}^{-1} =
\begin{bmatrix*}[c]
\bI &-\bB^{-1}\bL\\
&\bI
\end{bmatrix*}
\begin{bmatrix*}[c]
\bB^{-1}\\
&\bZ_{inv}^{-1}
\end{bmatrix*}
\begin{bmatrix*}[c]
\bI\\
-\bL^T\bB^{-1} & \bI
\end{bmatrix*}
=
\begin{bmatrix*}[c]
\bI &-\bB^{-1}\bL\\
&\bI
\end{bmatrix*}
\begin{bmatrix*}[c]
\bB^{-1}\\
-\bZ_{inv}^{-1}\bL^T\bB^{-1} &\bZ_{inv}^{-1}
\end{bmatrix*},
\end{equation}
where $\bZ_{inv}$ is a symmetric preconditioner of the Schur complement $\bS_{inv} = -\bL^T \bH^{-1} \bL = -\bA$. Notice that, by \cref{prop:Aprop}\ref{cor:Aspsd}, $\bS_{inv}$ is negative semidefinite. Hence, $\bbA$ is generally an indefinite matrix. Also, $\bbB_{inv}$ is positive definite, when $\bZ_{inv}$ is positive definite and indefinite otherwise. By \eqref{eq:specequiv}, depending on \eqref{eq:infsup}, $\bS_{inv}$ is spectrally equivalent to $-\bMM$. Observe that applying $\bbB_{inv}^{-1}$ requires two applications of $\bB^{-1}$ and two of $\bZ_{inv}^{-1}$.

Similarly, the following SPD preconditioner of the matrix $\bbA_{ss}$ in \eqref{eq:OSalgblocksystem} can be formulated:
\begin{equation}\label{eq:Bss}
\bbB_{ss}^{-1} =
\begin{bmatrix*}[c]
\bI &\bB_\omega^{-1}\bL\\
&\bI
\end{bmatrix*}
\begin{bmatrix*}[c]
\bB_\omega^{-1}\\
&\bZ_{ss}^{-1}
\end{bmatrix*}
\begin{bmatrix*}[c]
\bI\\
\bL^T\bB_\omega^{-1} & \bI
\end{bmatrix*}
=
\begin{bmatrix*}[c]
\bI &\bB_\omega^{-1}\bL\\
&\bI
\end{bmatrix*}
\begin{bmatrix*}[c]
\bB_\omega^{-1}\\
\bZ_{ss}^{-1}\bL^T\bB_\omega^{-1} &\bZ_{ss}^{-1}
\end{bmatrix*},
\end{equation}
where $\bB_\omega^{-1} = (\omega + 1)^{-1}\bB^{-1}$ and $\bZ_{ss}$ is a SPD preconditioner of the Schur complement $\bS_{ss} = \bMM - (\omega + 1)^{-1}\bA$. Note that $\bbA_{ss}$ is SPD and $\bS_{ss}$ is spectrally equivalent to $\bMM$ without requiring \eqref{eq:infsup} (i.e., even when $c_I = 0$) with the equivalence depending on $\omega$.

In \cref{ssec:precond}, we provide preliminary results with the above presented block preconditioners using $\bB^{-1} = \bH^{-1}$ and $\bZ_{ss} = \bI$, $\bZ_{inv} = -\bI$. We are interested in utilizing preconditioners based on algebraic multigrid methods \cite{TrottenbergMG,VassilevskiMG,rootnode} as $\bB^{-1}$, which will be investigated in a follow-up work.

\section{Application to linear hyperbolic problems}
\label{sec:application}

The considerations above are rather general. Here, we comment on certain particularities associated with the application of the methods to the hyperbolic problem \eqref{eq:pde}.

The differential operator in \eqref{eq:pde} can be written as $L = \hL + \sigma I$, where $\hL u = \gradd \bb u$. Then $L^\star = \hL^\star + \sigma I$, where integration by parts (Green's formula) \cite{GiraultFEM} implies $\hL^\star w = -\bb\cdot\grad w$. Thus, the PDE adjoint to \eqref{eq:pde} is also hyperbolic of similar type to \eqref{eq:pde}. In particular, when $\gradd \bb = 0$, then $\hL u = \bb \cdot \grad u$, i.e., $\hL^\star = -\hL$. Furthermore,
\begin{align*}
\cD(L) = \cD(\hL) &= \set{ u \in L^2(\Omega) \where \hL u \in L^2(\Omega) \text{ and } u = 0 \text{ on } \Gamma_I },\\
\cD(L^\star) = \cD(\hL^\star) &= \set{ w \in L^2(\Omega) \where \hL^\star w \in L^2(\Omega) \text{ and }  w = 0 \text{ on } \Gamma_O },
\end{align*}
where $\Gamma_O$ is the outflow portion of the boundary, $\Gamma_O = \set{\bx \in \partial \Omega \where \bn(\bx)\cdot\bb(\bx) > 0}$. Note that the boundary conditions in the definitions of $\cD(L)$ and $\cD(L^\star)$ make sense in terms of traces; see \cite{2004LinearHyperbolic}.

Under reasonable mild assumptions on $\bb$ a Poincar{\'e}-type inequality for $\hL^\star$ is shown in \cite[Lemma 2.4]{2004LinearHyperbolic}. A similar argument shows the respective inequality for $\hL$; cf., \cite[Lemma 6.8]{OlsonPhDthesis}. This covers the assumptions \eqref{asu:LstarPoincare}, \eqref{asu:LPoincare} (as well as \eqref{asu:Lstarsurjective}, by \cref{rem:equivasu}) for the case $\sigma \equiv 0$. The case of $\sigma \not\equiv 0$ is studied in \cite{2001ErrorLS}.

\section{Numerical results}
\label{sec:numerical}

Numerical results are shown in this section, which demonstrate the behavior of the methods presented and studied in this paper. Also, experiments with the block preconditioners of \cref{sec:implementation} are provided. The software used for implementing and testing the methods is FEniCS, cbc.block \cite{fenics}, PETSc \cite{petsc}, and LEAP (a least-squares package based on FEniCS that is under development at University of Colorado, Boulder).

\subsection{Experiments setting}

\begin{figure}
\centering
\includegraphics[width=0.5\textwidth]{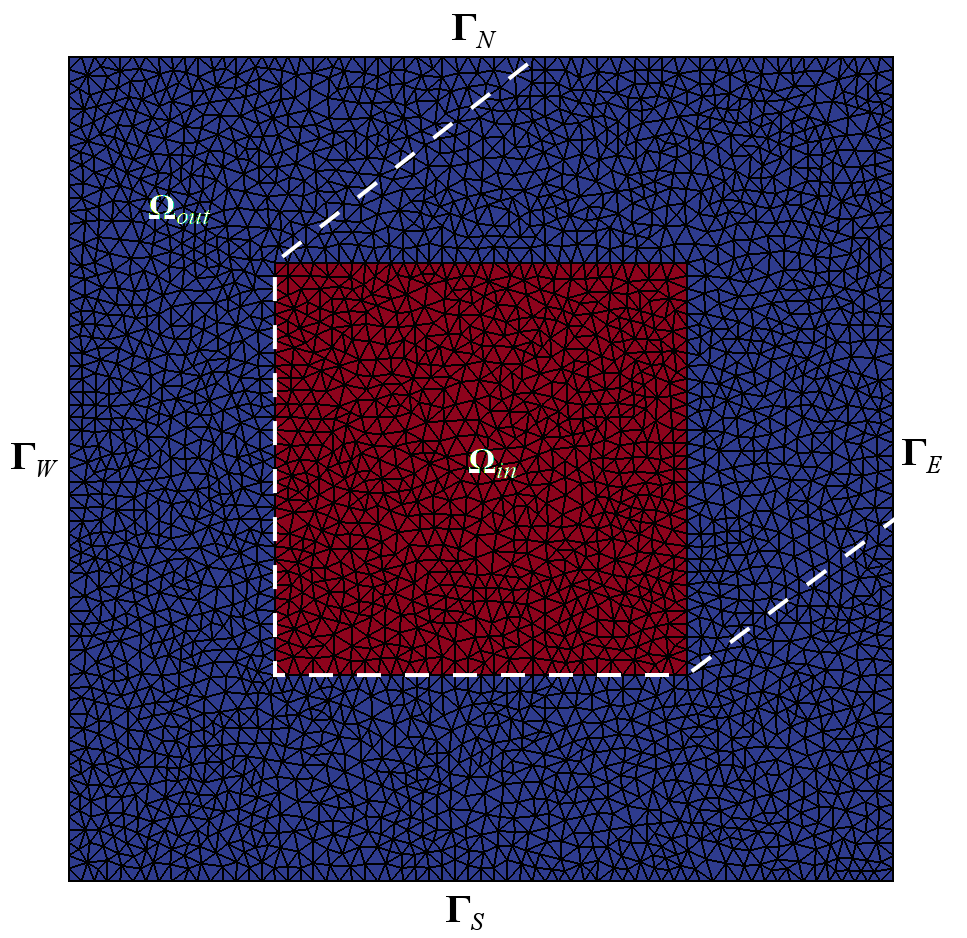}
\caption{Experiments setting.}\label{fig:setting}
\end{figure}

The domain, boundaries, structure of the coefficient $\sigma$, and a typical unstructured quasi-uniform triangular mesh (the coarsest mesh used in our experiments) are shown on \cref{fig:setting}. Namely, the domain is $\Omega = (0,1)^2$. It is split in two subregions -- $\Omega_{in} = (0.25, 0.75)^2$ and $\Omega_{out} = \Omega \setminus \Omega_{in}$. The coefficient $\sigma$ is taken discontinuous -- $\sigma = \sigma_{out}$ in $\Omega_{out}$ and $\sigma = \sigma_{in}$ in $\Omega_{in}$. We choose $\sigma_{out}$ small (i.e., $\Omega_{out}$ is a ``thin'' region) and $\sigma_{in}$ relatively large (i.e., $\Omega_{in}$ is a ``thick'' region). In particular, the experiments here use $\sigma_{out} = 10^{-4}$ and $\sigma_{in} = \set{10^4\text{ or }10}$. The choice $\sigma_{in} = 10^4$ provides a case when very steep exponential layers form, that are not well-resolved by the meshes. In contrast, when $\sigma_{in} = 10$, the exponential layers are less steep and can be resolved by a reasonably fine mesh. In all test cases, $\bb$ is taken $\bb = [\cos\alpha, \sin\alpha]$, where $\alpha = 3\pi/16$. Also, we set $r \equiv 0$ and $g = 1$ on $\Gamma_I$, where, with the current choice of $\bb$, $\Gamma_I = \Gamma_W \cup \Gamma_S$ and $\Gamma_O = \Gamma_E \cup \Gamma_N$. Thus, \eqref{eq:pde} becomes
\[
\begin{alignedat}{3}
\bb\cdot\grad \psi + \sigma \psi &= 0 &&\quad\text{in }\Omega,\\
\psi &= 1 &&\quad\text{on }\Gamma_I.
\end{alignedat}
\]

The dashed lines in \cref{fig:setting} show the locations of the exponential layers with the current $\bb$ and $\sigma$. The unstructured meshes do not follow the characteristics of the problem but they resolve the coefficient $\sigma$ (i.e., they conform to the subregions $\Omega_{in}$ and $\Omega_{out}$). In all tests, we set $\omega = 1$ in \eqref{eq:OSmin}.

\subsection{Convergence experiments}

In this subsection, the convergence, with respect to the $L^2(\Omega)$ norm, of the methods in this paper is demonstrated. In all cases, standard Lagrangian ($\cC^0$ piecewise polynomial) finite element spaces are utilized for $\cU^h$ and $\cZ^\fh$. Based on \cref{prop:Aprop}\ref{cor:Asingular}, it is always ensured that $\dim(\cZ^\fh) > \dim(\cU^h)$.

\begin{figure}
\centering
\subfloat[$\sigma_{in} = 10^4$]{\resizebox{0.49\textwidth}{!}{\input{conv_linquad_in1e4_out1e-4.texx}}\label{fig:conv_linquad_in1e4_out1e-4}}\;\;
\subfloat[$\sigma_{in} = 10$]{\resizebox{0.49\textwidth}{!}{\input{conv_linquad_in10_out1e-4.texx}}\label{fig:conv_linquad_in10_out1e-4}}
\caption{Convergence results. The spaces $\cU^h$ and $\cZ^\fh$ are on the same meshes, $\cU^h$ -- linear, $\cZ^\fh$ -- quadratic.}\label{fig:conv_linquad_out1e-4}
\end{figure}

\begin{figure}
\centering
\subfloat[$\cZ^\fh$ -- cubic]{\resizebox{0.49\textwidth}{!}{\input{conv_lincubic_in10_out1e-4.texx}}\label{fig:conv_linhigherorder_in10_out1e-4a}}\;\;
\subfloat[$\cZ^\fh$ -- quartic]{\resizebox{0.49\textwidth}{!}{\input{conv_linquart_in10_out1e-4.texx}}\label{fig:conv_linhigherorder_in10_out1e-4b}}
\caption{Convergence results. The spaces $\cU^h$ and $\cZ^\fh$ are on the same meshes, $\cU^h$ -- linear, $\sigma_{in} = 10$.}\label{fig:conv_linhigherorder_in10_out1e-4}
\end{figure}

\begin{figure}
\centering
\subfloat[$\sigma_{in} = 10^4$]{\resizebox{0.49\textwidth}{!}{\input{conv_linquint_in1e4_out1e-4.texx}}\label{fig:conv_linquint_in1e4_out1e-4}}\;\;
\subfloat[$\sigma_{in} = 10$]{\resizebox{0.49\textwidth}{!}{\input{conv_linquint_in10_out1e-4.texx}}\label{fig:conv_linquint_in10_out1e-4}}
\caption{Convergence results. The spaces $\cU^h$ and $\cZ^\fh$ are on the same meshes, $\cU^h$ -- linear, $\cZ^\fh$ -- quintic.}\label{fig:conv_linquint_out1e-4}
\end{figure}

\begin{figure}
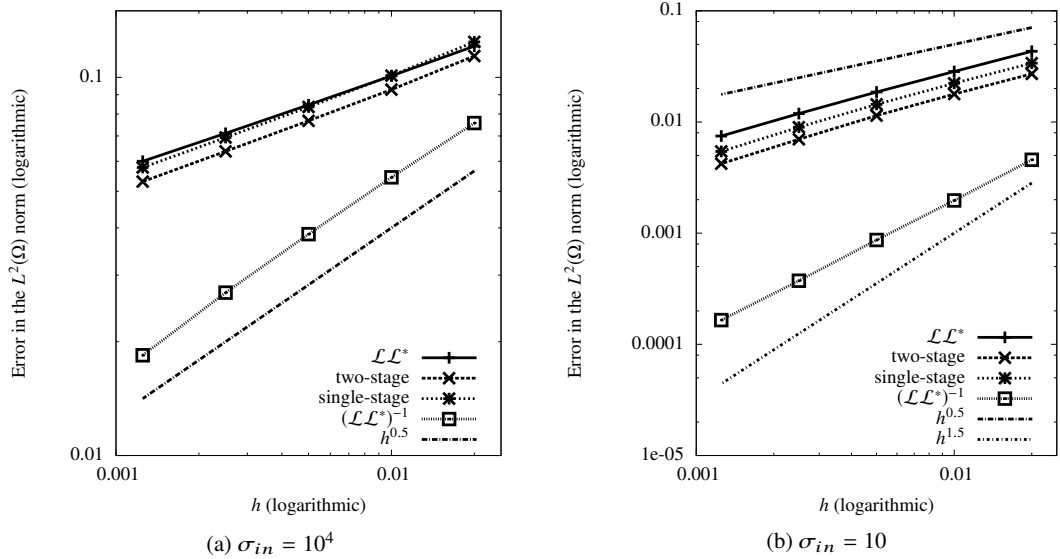

\centering
\subfloat[$\sigma_{in} = 10^4$]{\resizebox{0.49\textwidth}{!}{\input{conv_alllin_in1e4_out1e-4.texx}}\label{fig:conv_alllin_in1e4_out1e-4}}\;\;
\subfloat[$\sigma_{in} = 10$]{\resizebox{0.49\textwidth}{!}{\input{conv_alllin_in10_out1e-4.texx}}}
\caption{Convergence results. The spaces $\cU^h$ and $\cZ^\fh$ are both linear. Every mesh of $\cZ^\fh$ is obtained by a single uniform refinement of the respective $\cU^h$ mesh.}\label{fig:conv_alllin_out1e-4}
\end{figure}

First, results for $\cU^h$ -- linear, $\cZ^\fh$ -- quadratic, both spaces on the same respective meshes, and $\sigma_{in} = 10^4$ are shown on \cref{fig:conv_linquad_in1e4_out1e-4}. Note that, strictly speaking, the exact solution is in the Sobolev space $H^{3/2 - \epsilon}(\Omega)$, for any $\epsilon > 0$. According to the polynomial approximation theory \cite{BrennerFEM}, the optimal asymptotic rate of convergence, of the $L^2$-norm approximations of the exact solution on $\cU^h$, is $h^{3/2 - \epsilon}$. However, the analytical solution possesses very steep exponential layers that on the scale of the meshes behave like discontinuities (which is a case of interest). Therefore, intuitively, until the mesh begins resolving the exponential layers (i.e., before the ``asymptotic regime'' starts settling) the exact solution can, in a sense, be seen as ``discontinuous'', i.e., nearly behaving as a function in $H^{1/2 - \epsilon}(\Omega)$, and a rate of around $h^{1/2}$ can be considered as ``optimal'' initially with the potential of improving as the mesh is refined. More precisely, in view of the interpolation bounds of the polynomial approximation theory, the $H^{3/2 - \epsilon}(\Omega)$ norm of the analytical solution is rather large and this is associated with a delayed ``asymptotic regime'' of convergence. \Cref{fig:conv_linquad_in1e4_out1e-4} demonstrates that the $(\cL\cL^\star)^{-1}$ method obtains an $h^{1/2}$ rate. In comparison, the $\cL\cL^\star$, single-, and two-stage methods are slower to converge. Owing to \cref{thm:OSTSerrest}, this can be explained with the approximation properties of $L^\star(\cZ^\fh)$. It is interesting to notice that, in view of \cref{thm:OSTSerrest}, the single- and two-stage methods demonstrate slightly ``enhanced'' convergence rates compared to the $\cL\cL^\star$ method. The current theory cannot predict or explain such a behavior. It is unclear if this ``enhanced'' rate will be maintained once the ``asymptotic regime'' fully settles.

Note that the $h^{1/2}$ convergence of the $(\cL\cL^\star)^{-1}$ approximations in \cref{fig:conv_linquad_in1e4_out1e-4} does not necessarily mean that \eqref{eq:infsup} holds with $c_I$ independent of $h$. For example, observe \cref{fig:conv_linquad_in10_out1e-4}, which shows the same experiment as above but with $\sigma_{in} = 10$. That is, the exponential layers are now well-resolved by the meshes and the optimal asymptotic rate $h^{3/2 - \epsilon}$ is achievable. Notice that all methods, including the $(\cL\cL^\star)^{-1}$ method, demonstrate suboptimal rates. The $\cL\cL^\star$, single-, and two-stage methods this time converge with equal rates but slower than the $(\cL\cL^\star)^{-1}$ method and their respective errors are close to each other. The suboptimal convergence of the $(\cL\cL^\star)^{-1}$ approximations indicates that \eqref{eq:infsup} does not hold uniformly for this choice of spaces, i.e., $c_I$ in \eqref{eq:infsup} depends on $h$. \Cref{fig:conv_linquad_in10_out1e-4,fig:conv_linhigherorder_in10_out1e-4a,fig:conv_linhigherorder_in10_out1e-4b,fig:conv_linquint_in10_out1e-4} track the change (improvement) in the errors of the methods as the order of $\cZ^\fh$ is increased, for $\sigma_{in} = 10$.

Next, \Cref{fig:conv_linquint_out1e-4} (compare with \cref{fig:conv_linquad_out1e-4}) shows an experiment with quintic $\cZ^\fh$. Piecewise polynomial finite element spaces on triangles of order five (or higher) are special in the sense that they contain the space associated with the Argyris element; cf. \cite{BrennerFEM}. That is, $\cZ^\fh$ has a $\cC^1$ piecewise polynomial finite element subspace and $L^\star(\cZ^\fh)$ contains a $\cC^0$ piecewise polynomial finite element space. The results on \cref{fig:conv_linquint_out1e-4} suggest that this is not sufficient for the approximation properties of $L^\star(\cZ^\fh)$ to be on par with those of $\cU^h$, but \eqref{eq:infsup} may potentially hold. This is a subject of future investigation. Observe also that increasing the order of $\cZ^\fh$ in \cref{fig:conv_linquint_in1e4_out1e-4}, compared to \cref{fig:conv_linquad_in1e4_out1e-4}, results in improved errors for the $\cL\cL^\star$, single-, and two-stage methods, whereas this does not initially lead to an error improvement for the $(\cL\cL^\star)^{-1}$ method and only on finer meshes such an improvement can be observed. This creates the impression in \cref{fig:conv_linquint_in1e4_out1e-4} that the $(\cL\cL^\star)^{-1}$ solution converges with a rate higher than $h^\halff$. However, this is due to the sudden improvement in the size of the error, since the mesh is not sufficiently fine to resolve the steep layers and the asymptotic rate is not yet reached. That is, we are still in a pre-asymptotic regime and a rate around $h^\halff$ is to be expected, even from the actual best $L^2$-norm approximation on $\cU^h$.

The spaces $\cU^h$ and $\cZ^h$ do not need to be on the same mesh. This is demonstrated on \cref{fig:conv_alllin_out1e-4} for the case when $\cZ^h$ utilizes refined versions of the respective meshes of $\cU^h$. The results are very similar, with slightly slower rates, to those on \cref{fig:conv_linquad_out1e-4}.

Interestingly, in view of \cref{fig:conv_linquad_in1e4_out1e-4,fig:conv_linquint_in1e4_out1e-4,fig:conv_alllin_in1e4_out1e-4}, the losses of optimal rate ($h^{3/2 - \epsilon}$), caused by the unresolved exponential layers and the dependence of $c_I$ on $h$, do not seem to add up in the results for the $(\cL\cL^\star)^{-1}$ method. It seems that the slowest non-optimality dominates, which here is mostly the non-optimality of the mesh, and we obtain a rate of around $h^{1/2}$.

The methods in this paper target approximations in the $L^2(\Omega)$ norm and, as a result, provide a much better resolution of steep layers than a standard least-squares approach. However, more oscillations are now produced, which, due to the nature of the $L^2(\Omega)$ norm, do not prohibit convergence. Particularly, the $(\cL\cL^\star)^{-1}$ method produces substantially less oscillations than the $\cL\cL^*$-type methods and, interestingly, the $(\cL\cL^\star)^{-1}$ approach provides a slightly better resolution of steep layers, both contributing to smaller $L^2$-norm errors. This aligns with the observations that, in terms of solution quality, it is better to use $\cZ^\fh$ to approximate $(L_w L^*)^\mone$ than $L^*(\cZ^\fh)$ to approximate $\hu$ or, similarly, it is better to relate $\cU^h$ and $L^*(\cZ^\fh)$ via an inf-sup condition than via approximation properties. Particular plots of the solutions produced by the methods can be seen in \cite{PhDthesis}.

In our experiments, we observe that the local $L^2$-norm error in subregions away from the steep layers, where the solution is smooth, decreases with higher rates. Namely, in the case of a large contrast in $\sigma$ (i.e., corresponding to \cref{fig:conv_linquad_in1e4_out1e-4,fig:conv_linquint_in1e4_out1e-4}) the rate of local convergence is around $\cO(h)$ (where the optimal local rate for linear elements is $\cO(h^2)$) for all methods and both choices of $\cZ^\fh$ (quadratic and quintic). This demonstrates that the ``polluting'' effect of the steep layers is limited to some extent and requires further investigation.

Finally, observe that in the majority of the results above the $L^2$-norm errors of the single- and two-stage methods are smaller than the respective errors of the $\cL\cL^\star$ method. In all tests presented here, the two-stage method exhibits smaller errors compared to the single-stage method. In some cases, the convergence rates of the single- and two-stage approximations are ``enhanced'' (better) in comparison to the respective rates of the $\cL\cL^\star$ method. It is not completely clear if this ``enhanced'' error behavior is maintained asymptotically as $h \to 0$. Also, notice that as the order of $\cZ^\fh$ is increased, the error graphs of the different methods get more grouped together. This is predicted by the theoretical considerations in the previous sections.

\subsection{Preconditioning experiments}
\label{ssec:precond}

\begin{table}
\caption{Number of preconditioned GMRES(30) iterations for the $(\cL\cL^\star)^{-1}$ system \eqref{eq:algblocksystem} using relative tolerance $10^{-6}$ and the preconditioner $\bbB_{inv}^{-1}$ in \eqref{eq:Binv} with $\bB^{-1} = \bH^{-1}$ and $\bZ_{inv} = -\bI$.}\label{tbl:solver_inv_alllin_out1e-4}
\centering
\tabsize
\begin{tabular}{ l l r r c c }
\toprule
& & & & $\sigma_{in} = 10^4$ & $\sigma_{in} = 10$ \\
\cmidrule(lr){5-5}\cmidrule(lr){6-6}
$h$ & $\fh$ & $\dim(\cU^h)$ & $\dim(\cZ^\fh)$ & iterations & iterations \\
\midrule
0.02 & 0.01 & 3226 & 12645 & 73 & 51 \\
0.01 & 0.005 & 12645 & 50065 & 105 & 67 \\
0.005 & 0.0025 & 50065 & 199233 & 147 & 84 \\
0.0025 & 0.00125 & 199233 & 794881 & 201 & 106 \\
0.00125 & 0.000625 & 794881 & 3175425 & 255 & 119 \\
\bottomrule
\end{tabular}
\end{table}

\begin{table}
\caption{Number of preconditioned CG iterations for the single-stage system \eqref{eq:OSalgblocksystem} using relative tolerance $10^{-6}$ and the preconditioner $\bbB_{ss}^{-1}$ in \eqref{eq:Bss} with $\bB^{-1} = \bH^{-1}$ and $\bZ_{ss} = \bI$.}\label{tbl:solver_os_alllin_out1e-4}
\centering
\tabsize
\begin{tabular}{ l l r r c c }
\toprule
& & & & $\sigma_{in} = 10^4$ & $\sigma_{in} = 10$ \\
\cmidrule(lr){5-5}\cmidrule(lr){6-6}
$h$ & $\fh$ & $\dim(\cU^h)$ & $\dim(\cZ^\fh)$ & iterations & iterations \\
\midrule
0.02 & 0.01 & 3226 & 12645 & 30 & 25 \\
0.01 & 0.005 & 12645 & 50065 & 31 & 28 \\
0.005 & 0.0025 & 50065 & 199233 & 32 & 31 \\
0.0025 & 0.00125 & 199233 & 794881 & 34 & 34 \\
0.00125 & 0.000625 & 794881 & 3175425 & 35 & 35 \\
\bottomrule
\end{tabular}
\end{table}

Here, preliminary results with the block preconditioners of \cref{sec:implementation} are shown. In particular, we use $\bB^{-1} = \bH^{-1}$, $\bZ_{ss} = \bI$, and $\bZ_{inv} = -\bI$ to provide a basic idea on the behavior of the preconditioners. The effect of $\bB^{-1}$ is computed using a sparse direct solver -- MUMPS \cite{mumps}. All experiments in this subsection are for the case when $\cU^h$ and $\cZ^\fh$ are piecewise linear and $\cZ^\fh$ uses meshes that are obtained from the respective $\cU^h$ meshes by a single uniform refinement, i.e., $\fh = h/2$. That is, the results here correspond to \cref{fig:conv_alllin_out1e-4}. In all tests, the iterative processes are stopped when the overall relative reduction of the norm of the preconditioned residual becomes less than $10^{-6}$.

\Cref{tbl:solver_inv_alllin_out1e-4} shows the number of preconditioned GMRES(30) \cite{SaadNLA,VorstNLA} iterations for the $(\cL\cL^\star)^{-1}$ system \eqref{eq:algblocksystem} using the block preconditioner $\bbB_{inv}^{-1}$ in \eqref{eq:Binv}. Observe that $\bbB_{inv}^{-1}$ is not optimal with the choice $\bZ_{inv} = -\bI$. As discussed in \cref{sec:implementation}, this can be associated with the dependence of $c_I$, in \eqref{eq:infsup}, on the mesh parameter, $h$, and as a result the spectral relation \eqref{eq:specequiv} does not hold uniformly (i.e., it depends on $h$). This suggests that further care is necessary in preconditioning the Schur complement $\bS_{inv}$ and the simple choice $\bZ_{inv} = -\bI$ is insufficient in this case. We plan to further investigate this, together with the utilization of algebraic multigrid as $\bB^{-1}$, in a follow-up work.

In contrast, as discussed in \cref{sec:implementation}, the block preconditioner $\bbB_{ss}^{-1}$ in \eqref{eq:Bss} is optimal for the single-stage system \eqref{eq:OSalgblocksystem} with the choice $\bZ_{ss} = \bI$, independently of \eqref{eq:infsup}. Also, preconditioned CG can be used in this case. Results are shown in \cref{tbl:solver_os_alllin_out1e-4}.

\section{Conclusions and further development}
\label{sec:conclusions}

We proposed the $(\cL\cL^\star)^{-1}$, together with the $\cL\cL^\star$, single-, and two-stage, methods and studied their application to scalar linear hyperbolic PDEs, aiming at obtaining $L^2$-norm approximations on finite element spaces. Error estimates were shown, pointing to the factors that affect the convergence and providing conditions that guarantee optimal rates. Also, numerical results were demonstrated. The methods clearly show $L^2$-norm convergence and often with acceptable rates. The $(\cL\cL^\star)^{-1}$ method demonstrates the best convergence rates but it induces the most difficult linear systems to solve.

The considerations in this paper suggest further directions of research. A few of them are mentioned in the exposition and in \cref{sec:extension}. Some additional topics are the following: combining the $(\cL\cL^\star)^{-1}$ formulation with FOSLS terms in a ``hybrid'' method; the potential of using $\cZ^\fh$ on different meshes (even if we have no freedom to choose the mesh of $\cU^h$, we can select $\cZ^\fh$ freely) that can be better tailored to the particular problem and, thus, obtain $(L_w L^\star)_\fh^{-1}$ that better approximates $(L_w L^\star)^{-1}$, in some sense. The inf-sup condition and its relation to the approximation properties of the finite element spaces is an interesting and very challenging topic. This would allow further comparison between the methods in terms of the derived error estimates. Currently, the numerical results and basic analysis suggest that the requirements on the approximation properties of $L^\star(\cZ^\fh)$ may possibly be stronger than the inf-sup condition. At least, we observe that when $L^\star(\cZ^\fh)$ provides neither on par approximation properties, nor a uniform inf-sup condition, then the  $(\cL\cL^\star)^{-1}$ method seems less affected by the deficiencies of $L^\star(\cZ^\fh)$ in terms of convergence rates, but it may suffer more in terms of the efficiency of the linear solver. Furthermore, it is intriguing to study the influence of the coefficient $\sigma$ on the constant in \eqref{eq:infsup} and, thus, on the behavior of the method, as well as whether and how the Poincar{\'e} constants in \eqref{asu:LstarPoincare}, \eqref{asu:LPoincare} affect \eqref{eq:infsup}.

The proposed block preconditioner is one approach to solving the $(\cL\cL^\star)^{-1}$ and single-stage linear systems. It would be interesting to study the adaptation and utilization of other methods, developed for ``saddle-point problems''. Preconditioning the matrix $\bH$, coming from hyperbolic operators, $L^\star$, also suggests further development, which is applicable beyond the methods of this paper.

\appendix
\section{Generalizing the formulations}
\label{sec:extension}

In this appendix, for completeness, we review possible generalizations and extensions of the formulations in this paper. In particular, we discuss the ``weak'' treatment of the inflow boundary condition and the potential of utilizing general (possibly discontinuous) finite element spaces as $\cU^h$. Considering $\cU^h \subset \cD(L)$, for the $(\cL\cL^\star)^{-1}$, single-, and two-stage methods, and enforcing the boundary data by superposition corresponds to imposing the boundary condition ``strongly''. For the general case, when $\cU^h \subset L^2(\Omega)$ is possibly piecewise discontinuous, it is necessary to impose the boundary condition in a ``weak'' sense (i.e., as a part of the variational formulation).

Recall that, for simplicity, the hyperbolic problem \eqref{eq:pde} was reformulated as the operator equation \eqref{eq:opeq} using superposition to enforce the boundary data $g$ on $\Gamma_I$. In particular, this simplifies the weak formulation associated with the $\cL\cL^\star$ minimization \eqref{eq:stdLLstarmin}, since the exact solution, $\hu$, of \eqref{eq:opeq} is in  $\cD(L)$ (i.e., $\hu = 0$ on $\Gamma_I$). Alternatively, consider the original PDE \eqref{eq:pde} and let $\hpsi$ denote its exact solution. The respective $\cL\cL^\star$ minimization is
\begin{equation}\label{eq:stdLLstarminBC}
\xi^\fh_\star = \argmin_{w^\fh \in \cZ^\fh} \lV L^\star w^\fh - \hpsi \rV^2.
\end{equation}
The resulting $\cL\cL^\star$ approximation is $\psi^\fh_\star = L^\star \xi^\fh_\star = \Pi^\fh_\star \hpsi \in L^\star(\cZ^\fh)$. Using integration by parts (Green's formula), the weak form corresponding to \eqref{eq:stdLLstarminBC} is the following:
\[
\text{Find } z^\fh \in \cZ^\fh\col \li L^\star z^\fh, L^\star w^\fh \ri = \li r, w^\fh \ri - \int_{\Gamma_I} (\bb g)\cdot\bn\; w^\fh \dd\sigma, \quad \forall w^\fh \in \cZ^\fh.
\]
Note that only the right hand side is different and it involves only given data.

Using the $\cL\cL^\star$ principle \eqref{eq:stdLLstarminBC} leads to minor changes in the single- and two-stage formulations. Indeed, it is sufficient to replace $\brff$ with $\brff_b$ and $\bff$ with $\bff_b = \bL^T \bH^{-1} \brff_b \in \bbR^N$ in the respective linear systems above, where $\brff_b \in \bbR^M$ is defined as
\[
(\brff_b)_i = \li r, \psi^\fh_i \ri - \int_{\Gamma_I} (\bb g)\cdot\bn\; \psi^\fh_i \dd\sigma.
\]
Note that this can be combined with a ``strong'' enforcement of the boundary data on $\cU^h$ by standard means of the finite element methods, which demonstrates the flexibility that least-squares often provide. The analysis in \cref{sec:llstar} remains valid. Furthermore, the single- and two-stage methods are clearly general enough and allow the utilization of general finite element spaces $\cU^h \subset L^2(\Omega)$.

As discussed in \cref{sec:properties}, $L_w$ is, in a sense, an extension of the operator $L$ (more precisely, of $\cE L$) on $L^2(\Omega)$. Therefore, the $(\cL\cL^\star)^{-1}$ formulation is extended to general spaces $\cU^h \subset L^2(\Omega)$ by replacing the operator $L$ with its ``weak'' version $L_w$. This idea is already applied, for theoretical purposes, in the proof of \cref{thm:errest} and in \cref{sec:implementation}, when considering the operator $L^\star B_\fh^{-1} L_w$, since the bilinear forms need to be defined on the whole of $L^2(\Omega)$ to obtain error estimates with respect to the $L^2(\Omega)$ norm. In fact, the generalized $(\cL\cL^\star)^{-1}$ method is precisely the one in \eqref{eq:minLLstarinvproj} and \eqref{eq:LLstarinvproj}. Note that this is not only a tool of analysis but results in feasible formulations. Indeed, the weak formulation \eqref{eq:LLstarinvsystem} is already stated in such a general form with the exception of the right hand side, which needs to be modified to accommodate the ``weak'' enforcement of the boundary condition. Again, it is sufficient to replace $\brff$ with $\brff_b$ and $\bff$ with $\bff_b$ in the respective linear systems. Clearly, the analysis in this paper remains valid.
Similarly, the modified (by a preconditioner) formulation \eqref{eq:modmin} can be extended to general finite element spaces $\cU^h \subset L^2(\Omega)$.

In summary, we observed that the $(\cL\cL^\star)^{-1}$, single-, and two-stage methods can be easily generalized to arbitrary finite element spaces $\cU^h \subset L^2(\Omega)$. However, a piecewise discontinuous space $\cU^h$ (which is a case of interest) is rather rich, whereas $L^\star(\cZ^\fh)$ is constrained by requiring $\cZ^\fh \subset \cD(L^\star)$. This poses further difficulties in maintaining \eqref{eq:infsup} or on  par approximation properties of $L^\star(\cZ^\fh)$, regarding the estimate in \cref{thm:OSTSerrest}, when $\cU^h$ is discontinuous. Removing or reducing the constraint $\cZ^\fh \subset \cD(L^\star)$ is a challenging topic and a subject of future work.

Finally, the considerations in this paper are rather general. In the exposition above, for simplicity of notation and since the scalar PDE \eqref{eq:pde} is considered, only $L^2(\Omega)$ is used. Nevertheless, in general, $L$ may come either from a scalar PDE or a first-order system of PDEs. In the latter case, the considerations in this paper can be extended to systems as long as the occurrences of $L^2(\Omega)$ are replaced by the appropriate product $L^2$ spaces with their respective product $L^2$ norms and the assumptions are satisfied. This also suggests a subject of further investigations.

\bibliographystyle{wileyj}
\bibliography{references_finalrev}
\end{document}